\documentclass[english]{article}
\usepackage{geometry}
\usepackage{a4wide}
\usepackage{amsmath}
\usepackage{amsfonts}
\usepackage{amsthm}
\usepackage{amssymb}
\usepackage{graphicx}
\usepackage{color}
\usepackage{booktabs}
\usepackage{underscore}

\newtheorem{lemma}{Lemma}

\newtheorem{theorem}{Theorem}

\newcommand{\fs}{{\frak{s}}}

\title{Many Server Scaling of the N-System Under FCFS-ALIS}
\author{
Dongyuan Zhan\thanks{School of Management, University College London, Gower Street, London WC1E 6BT, United Kingdom; email: {\em d.zhan@ucl.ac.uk}}
\and
Gideon Weiss\thanks{
Department of Statistics,
The University of Haifa,
Mount Carmel 31905, Israel; email:
{\em gweiss@stat.haifa.ac.il}
Research supported in part by
Israel Science Foundation Grants 711/09 and 286/13.
}
}
\date{\today}
\usepackage{babel}
\begin{document}
\maketitle

\begin{abstract}
The N-System with independent Poisson arrivals and exponential server-dependent service times under first come first served and assign to longest idle server policy has explicit steady state distribution. We scale the arrival and the number of servers simultaneously, and obtain the fluid and central limit approximation for the steady state. This is the first step towards exploring the many server scaling limit behavior of general parallel service systems. 
\end{abstract}

\section{Introduction}
\label{sec.introduction}
In this paper we study the many server N-System shown in Figure \ref{fig.N}, with Poisson arrivals and exponential service times, under first come first served and assign to longest idle server policy (FCFS-ALIS), as the number of servers becomes large.  Before describing the model in detail, we will first discuss our motivation for studying this system.
  \begin{figure}[htb]
  \begin{center}
  \includegraphics[width=0.20\linewidth]{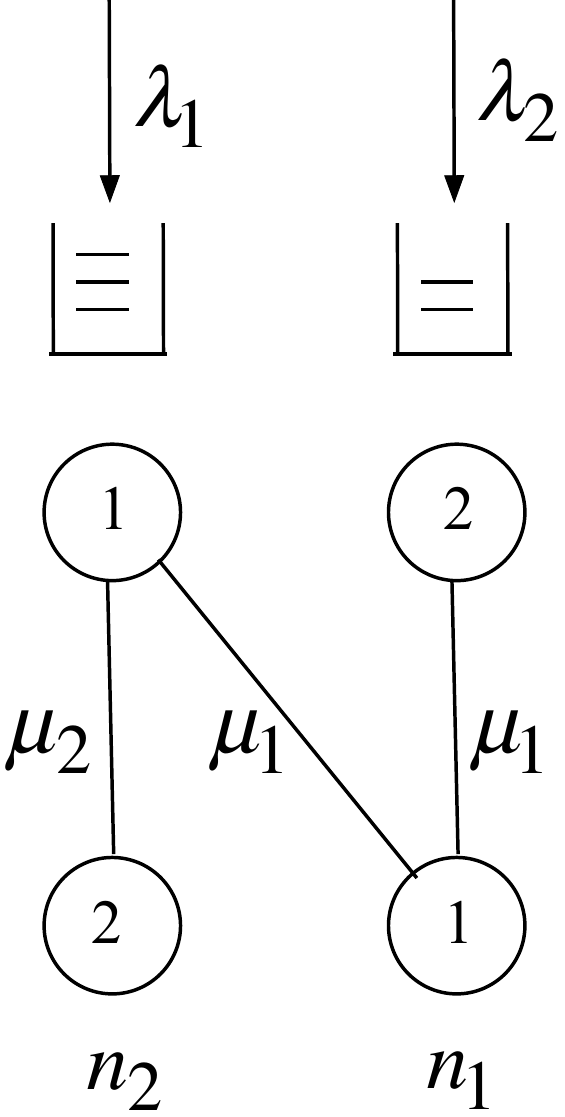}
  \end{center}
  \caption{The  multi-server N-System}
  \label{fig.N}
  \end{figure}

The N-System is one of the simplest special cases of the so called parallel server systems, as defined in \cite{foss-chernova:98,harrison-lopez:99} and further studied in \cite{harchol-balter-etal:99,williams:00,bell-williams:01,rubino-ata:09,armony-ward:10,armony-ward:13,gurvich-whitt:09,gurvich-whitt:10,adan-boon-weiss:14}.
The general model has customers of types $i=1,\ldots,I$,  servers of types $j=1,\ldots,J$, and a bipartite compatibility graph $G$ where $(i,j)\in G$ if customer  type $i$ can be served by server $j$.  Arrivals are renewal with rate $\lambda$, where successive customer types are i.i.d. with probabilities $\alpha_i$, there is a total of $n$ servers, of which $n \theta_j$ are of type $j$, and service times are generally distributed with rates $\mu_{i,j}$.  Assume the system is operated under the FCFS-ALIS policy, that is servers take on the longest waiting compatible customer, and arriving customers are assigned to the longest idle compatible server.  For this general system necessary and sufficient conditions for stability (positive Harris recurrence for given $\lambda$), or for complete resource pooling (there exists critical $\lambda_0$ such that the system is stable for $\lambda<\lambda_0$, and the queues of all customer types diverge for $\lambda>\lambda_0$) cannot be determined by 1st moment information alone (as shown by an example of Foss and Chernova \cite{foss-chernova:98}).  In particular, under FCFS-ALIS  calculation of the matching rates $r_{i,j}$, which are long term average fractions of services performed by servers of type $j$ on customers of type $i$, is intractable.

In the special case that service rates depend only on the server type, and not on customer type, with Poisson arrivals and exponential service times,  the system has a product form stationary distribution, as given in \cite{adan-weiss:14}.  In that case matching rates can be computed from the stationary distribution.

The following conjecture was made in \cite{adan-boon-weiss:14}: if the system is stable and has complete resource pooling for given $\lambda,\,n$, and we let both become large together, the behavior of the system simplifies:  there will exist $\beta_j$ such that servers of type $j$ perform a fraction $\beta_j$ of the services, and the matching rates $r_{i,j}$  will  converge to the rates for the FCFS infinite matching model with $G,\alpha,\beta$, as calculated in \cite{adan-weiss:11} (see also \cite{adan-busic-mairesse-weiss:15}).
The conjecture is based on the following heuristic argument:  in steady state the times that each server becomes available form a stationary process which is only mildly correlated with the other servers and so servers become available approximately as a superposition of almost independent stationary processes which in the many server limit becomes a Poisson process, and server types are then i.i.d. with probabilities $\beta_j$, while customer types arrive as an i.i.d. sequence with probabilities $\alpha_i$, which corresponds exactly to the model of FCFS infinite matching.

In our current study of the many server N-System we shall verify the conjectured many server  behavior for this simple parallel server system.  To do so we start from the known stationary distribution of the N-System with many servers, as derived from \cite{adan-weiss:14}, and study its behavior as $n\to\infty$.  As it turns out, the product form stationary distribution even for this simple case is far from simple, and the derivations of limits, which use summations over server permutations and asymptotic expansions of various expressions are quite laborious.  We feel that this emphasizes the difficulty of verifying the conjectured behavior of the general system, which remains intractable at this time.

We mention that the  N-System with just two servers, has been the subject of several papers, \cite{green:85,adan-foley-mcdonald:07,ghamami-ward:13,tezcan-dai:10}. In this paper, our focus is on the N-System with many servers under FCFS-ALIS policy, and its limit property.

The rest of the paper is structured as follows:  In Section \ref{sec.model} we describe the model, in Section \ref{sec.fluid} we use some heuristic arguments to obtain a guess at the limiting behavior.  In Section \ref{sec.stationary} we obtain the stationary behavior under many server scaling.  
In Section \ref{sec.numerics} we illustrate our results with some numerical examples.
To improve the readability of the paper we have put all the proofs for Section \ref{sec.stationary} in the appendix.


\section{The Model}
\label{sec.model}
In our N-System customers of types $c_1$ and $c_2$ arrive as independent Poisson streams,  with rates $\lambda_{1},\lambda_{2}$.   There are skill based parallel servers, $n_1$ servers of type $s_1$ which are flexible and can serve both types, and $n_2$  servers of type $s_2$ which can only serve type $c_1$ customers.  We assume service times are all independent exponential, with server dependent rates. The service rate of an $s_1$ server is $\mu_1$, the service rate of an $s_2$ server is $\mu_2$, see Fig \ref{fig.N}.   We let $\lambda=\,\lambda_1+\,\lambda_2,\,n=n_1+n_2$.  Service policy is FCFS-ALIS. 

The system is obviously Markovian.  In  \cite{adan-foley-mcdonald:07,adan-visschers-weiss:12,adan-weiss:14} the following state description for the server dependent Poisson exponential system with $J$ server types and $I$ customer types was used:  
imagine the customers arranged in a single queue by order of arrivals, and servers are attached to customers which they serve, and the remaining idle servers are arranged by increasing idle time, see Figure \ref{fig:ivostate}
  \begin{figure}[htb]
  \begin{center}
  \includegraphics[width=0.90\linewidth]{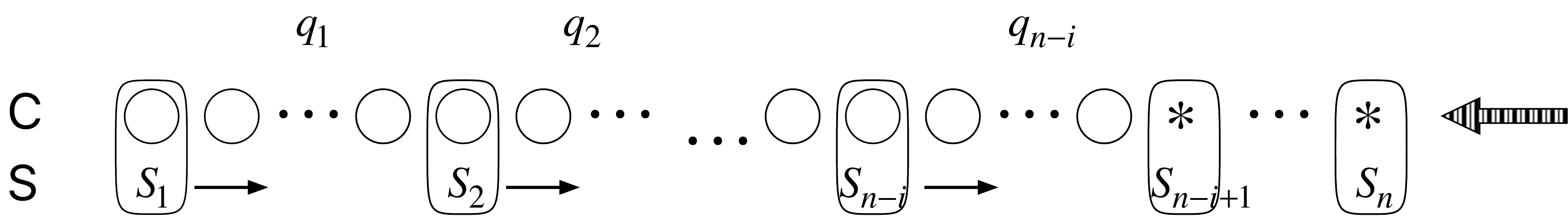}
  \end{center}
  \caption{state description under FCFS-ALIS}
  \label{fig.ivostate}
  \end{figure}
  The state is then $\frak{s}=(S_1,q_1,S_2,q_2,\ldots,S_{n-i},q_{n-i},S_{n-i+1},\ldots,S_n)$, where $S_1,\ldots,S_n$ is a permutation of the $n$ servers, the first $n-i$ servers are the ordered  busy servers, and the last $i$ servers are the ordered idle servers, and where $q_j,\,j=1,\ldots, n-i$ are the queue lengths of the customers waiting for one of the servers $S_1,\ldots, S_j$, and skipped (cannot be served) by  servers $S_{j+1},\ldots,S_n$.

 For the special case of the N-System, the following three random quantities are important:
  $i_1=I_1(\frak{s})$  the number of idle servers of type $s_1$,  $i_2=I_2(\frak{s})$  the number of idle servers of type $s_2$,  and  $k=K(\frak{s})\geq 0$  the number of servers of type $s_2$ which follow the last server of type $s_1$ in the sequence $S_1,\ldots,S_n$.  We let $i=I(\frak{s})$ be the total number of idle servers.
Because of the structure of the N-System, and the FCFS-ALIS policy the following properties hold for $i=0,\ldots,n$ and $k=0,\ldots,n_2$:
\begin{description}
\setlength{\itemsep}{3pt}
\item[(i)]
There are no customers waiting for any server which precedes the last  $s_1$ server in the permutation. In other words, for all $j < \min (n-k, n-i)$ we have $q_j=0$.  In particular, if there is an idle type $s_1$ server, in other words if $i > k$, then there are no waiting customers at all.
\item[(ii)]
If there are any idle servers, then there are no type $c_1$  customers waiting for service, in other words, if $i>0$ then all the waiting customers are of type $c_2$.
\item[(iii)]
If there are no idle servers, then only the last queue can contain type $c_1$ customers, in other words, if $i=0$ then the last queue may contain customers of both types, but all the other waiting customers are of type $c_2$.
\end{description}

Denote
\[
\alpha = \frac{\lambda_1}{\lambda},  \quad \theta=\frac{n_1}{n},  \quad \rho=\frac{\lambda}{n_1\mu_1+n_2\mu_2},  \quad  \delta=\frac{\lambda_2}{n_1 \mu_1}.
\]
Then a necessary and sufficient condition for stability is
\[
\rho <1 , \quad  \delta<1
\]
We shall require a stronger condition of complete resource pooling, defined by
\[
\alpha + \beta > 1
\]
where $\beta$ equals the long run fraction of services performed by $s_1$ servers.  The value of $\beta$ will be calculated in the next Section.

Using the results of \cite{adan-weiss:11,adan-weiss:14} we can then write the exact  stationary distribution of this system.   We wish to show that as   the arrival rates and the number of servers increase the system simplifies, and we get very precise many server scaling limits.
We will investigate the behavior of the system when $\alpha,\theta,\rho$ are fixed, and $n \to \infty$.
To be  precise, we shall then have $n, \lambda=\rho n,\,\lambda_1=\alpha \lambda,\,\lambda_2=(1-\alpha)\lambda$, $n_1=\lceil \theta n \rceil,\,n_2 = n-n_1$, all of which go to $\infty$.

\section{Fluid Calculations}
\label{sec.fluid}
We perform the following heuristic calculation:  As long as the system is underloaded ($\rho<1$), each server of type $s_1$ will have a cycle of service of mean length  $1/\mu_1$, followed by an idle period, and similarly each server of type $s_2$ will have service of mean length $1/\mu_2$ followed by an idle period.  The key idea now is that when $n\to\infty$, the idle periods should have the same length for both types, because of ALIS.  Let $T$ be the average length of the idle time.  The average cycle times will be:  $1/\mu_1+T$ and $1/\mu_2+T$.  Denote by $\beta$  the long run fraction of services performed by $s_1$ servers, and $1-\beta$ for type $s_2$.  The flow rate out of one type $s_1$ server is $1/(1/\mu_1+T)$, the flow rate out of all type $s_1$ servers should equal $\lambda\beta$. Similarly the flow rate out of all type $s_2$ servers should equal $\lambda(1-\beta)$. That is, 
\[
\lambda \beta = n_1/(1/\mu_1+T),     \qquad  \lambda (1-\beta)=n_2/(1/\mu_2+T).
\]
Now we solve for $T$ and $\beta$:  we rewrite
\[
\beta = \frac{n_1}{\lambda} \frac{1}{1/\mu_1+T}, \qquad  1- \beta = \frac{n_2}{\lambda} \frac{1}{1/\mu_2+T}
\]
and eliminate $\beta$:
\begin{equation*} 
1 = \frac{n_1}{\lambda} \frac{1}{1/\mu_1+T} +  \frac{n_2}{\lambda} \frac{1}{1/\mu_2+T}
\end{equation*}
to get a quadratic equation for $T$:
\[
g(T)= \lambda \mu_1 \mu_2 T^2 +\big( \lambda(\mu_1+\mu_2) - (n_1+n_2) \mu_1 \mu_2  \big) T
+ \lambda -n_1 \mu_1 - n_2 \mu_2 =0.
\]
Here $g(0)<0$ by  $\rho <1$, so the equation has one positive and one negative root.  Solving for positive $T$ we get:
\begin{equation}\label{eqn.Tsolution}
 \begin{aligned}
 T &=& \frac{1}{2} \left( \frac{n}{\lambda} -\frac{1}{\mu_1} - \frac{1}{\mu_2} +
 \sqrt{\frac{n^2}{\lambda^2} + 2 \,\frac{n_1-n_2}{\lambda}\Big(\frac{1}{\mu_1} - \frac{1}{\mu_2}\Big)
 + \Big(\frac{1}{\mu_1} - \frac{1}{\mu_2}\Big)^2} \right)
 \end{aligned}
\end{equation}
 Note:  For the case of $\mu_1 = \mu_2 = \mu$ we get $T=\frac{1-\rho}{\rho}\frac{1}{\mu}$.

From $T$ and little's law we can obtain the average number of idle servers in pool 1 and pool 2, denoted by $m_1$ and $m_2$ respectively.
\begin{equation}\label{eqn.MeanIdle}
m_1 = T\lambda\beta = \frac{Tn_1}{T+1/\mu_1},\qquad m_2 = T\lambda(1-\beta)=\frac{Tn_2}{T+1/\mu_2}.
\end{equation}
 The values of $\beta$ and $1-\beta$ are then:
 \[
 \beta = \frac{n_1}{\lambda T+\lambda/\mu_1}, \qquad
 1- \beta = \frac{n_2}{\lambda T+\lambda/\mu_2}
 \]
Note:  both are positive, so $0<\beta<1$.  Also, when $\mu_1=\mu_2$ we get $\beta=\theta$.

The value of  $\alpha$ does not come into the equation for $T$, or the calculation of $\beta$.
Hence, once we solve and obtain $\beta$, the property of complete resource pooling will consist of  checking that $\alpha > 1 - \beta$.

We will show that the following holds for the stationary queue, as $n\to\infty$:
\begin{itemize}
\item
$K(\frak{s})$ is distributed as a geometric random variable, taking values $0,1,2,\ldots$ with probability of success $1-\frac{1-\beta}{\alpha}$.  It is independent of $I_1(\frak{s}),I_2(\frak{s})$.
\item
$\left(I_1(\frak{s}),\,I_2(\frak{s})\right)$ is close to a bivariate Normal, with means $(m_1,m_2)$,  variances
\[
\left(\frac{(n_1-m_1)m_1(n_2m_1+m_2^2)}{n_1m_2^2+n_2m_1^2},\, \frac{(n_2-m_2)m_2(n_1m_2+m_1^2)}{n_1m_2^2+n_2m_1^2}\right),
\]
and correlation
\[
\left(\frac{(n_1-m_1)(n_2-m_2)m_1m_2}{(n_1m_2+m_1^2)(n_2m_1+m_2^2)}\right)^{\frac{1}{2}}.
\]
\item
Successive idle servers except for the last $K+1$ are i.i.d. of type $s_1$ with probability $\beta$ and of type $s_2$ with probability $1-\beta$.
\end{itemize}

\section{Many server limit of the stationary distribution}
\label{sec.stationary}

\subsection{Exact Stationary Distributions}
We first obtain the stationary distribution for each state $\frak{s}$.  We note that the stationary probabilities depend mainly on the values of $K(\frak{s}),\,I_1(\frak{s}),\,I_2(\frak{s})$. Let $\mu(S_j)$ denote the service rate of the server at position $j$.

\begin{theorem}
\label{thm.markov}
The stationary distribution of the state $\frak{s}$ of the FCFS-ALIS many server N-system is given by:
\begin{equation}\label{eqn.pistate}
\pi(\frak{s}) = \left\{ \begin{array}{ll}
\displaystyle
B \prod_{l=1}^{n-i_1-i_2} \left( \sum_{j=1}^l  \mu(S_j) \right)^{-1} \left(\frac{1}{\lambda}\right)^{i_1+i_2-k} \left(\frac{1}{\lambda_1}\right)^{k}, &
\begin{array}{l}
k=0,\ldots,n_2,\;  \\
i_1=1,\ldots,n_1,\; \\ i_2=k,\ldots,n_2,
\end{array} \\
\\
\displaystyle
B \prod_{l=1}^{n-k-1} \left( \sum_{j=1}^l  \mu(S_j) \right)^{-1} \prod_{j=n-k}^{n-i_2} \frac{\lambda_2^{q_j}}{(\mu_1n_1+\mu_2(j-n_1))^{q_j+1}}\;
\left(\frac{1}{\lambda_1}\right)^{i_2}, \quad &
\begin{array}{l}
k=1,\ldots,n_2,\; \\ i_1=0,\;\\ i_2=1,\ldots,k,
\end{array} \\
\\
\displaystyle
B \prod_{l=1}^{n-k-1} \left( \sum_{j=1}^l  \mu(S_j) \right)^{-1} \prod_{j=n-k}^{n-1} \frac{\lambda_2^{q_j}}{(\mu_1n_1+\mu_2(j-n_1))^{q_{j}+1}}\;
\frac{\lambda^{q_n}}{(\mu_1n_1+\mu_2n_2)^{q_n+1}}, &
\begin{array}{l}
k=0,\ldots,n_2,\;\\ i_1=i_2=0.
\end{array}
\end{array}\right.
\end{equation}
where $B$ is a normalizing constant.
\end{theorem}
\begin{proof}
This follows for all three  parts of (\ref{eqn.pistate}) by utilizing properties (i),(ii),(iii) in Section~\ref{sec.model} and substituting into Equation (2.1), Theorem 2.1, in \cite{adan-weiss:14}.
\end{proof}

Before we manipulate equation (\ref{eqn.pistate}), we introduce a lemma to facilitate the calculation.
\begin{lemma}\label{thm.sum}
Let $A_1,\ldots,A_m$ denote a permutation of $m$ given positive real numbers $a_1,\ldots,a_m$, we have
\[
\sum_{(A_1,\ldots,A_m)\in\mathcal{P}(a_1,\ldots,a_m) } \prod_{l=1}^m \left( \sum_{j=1}^l  A_j \right)^{-1} =\left( \prod_{l=1}^m  a_l \right)^{-1}
\]
where $\mathcal{P}(a_1,\ldots,a_m)$ denotes the set of all the permutations of $a_1,\ldots,a_m$.
\end{lemma}

Now we can get the joint stationary distribution of $K(\frak{s}),\,I_1(\frak{s}),\,I_2(\frak{s})$. We denote by $\pi(k,i_1,i_2)$ the stationary probability of $K(\frak{s})=k$, $I_1(\frak{s})=i_1$ and $I_2(\frak{s})=i_2$.
\begin{theorem}
\label{thm.piI1I2K}
The steady state joint distribution of $K(\frak{s}),\,I_1(\frak{s}),\,I_2(\frak{s})$ is given by:
\begin{equation}
\label{eqn.piKI}
\pi(k,i_1,i_2) = \left\{ \begin{array}{ll}
\displaystyle
B_1 {n_1\choose i_1}{n_2\choose i_2} \frac{i_1 i_2! (i_1+i_2-k-1)!}{(i_2-k)!}\mu_1^{i_1}\mu_2^{i_2} \left(\frac{1}{\lambda}\right)^{i_1+i_2} \left(\frac{\lambda}{\lambda_1}\right)^{k}, &
\begin{array}{l}
k=0,\ldots,n_2,\;  \\
i_1=1,\ldots,n_1,\; \\ i_2=k,\ldots,n_2,
\end{array} \\
\\
\displaystyle
B_1 \frac{n_1\,n_2!}{(n_2-k)!}\mu_1\mu_2^{k} \prod_{j=n-k}^{n-i_2} \frac{1}{\mu_1n_1+\mu_2(j-n_1)-\lambda_2}\;
\left(\frac{1}{\lambda_1}\right)^{i_2}, \quad &
\begin{array}{l}
k=1,\ldots,n_2,\; \\ i_1=0,\;\\ i_2=1,\ldots,k,
\end{array} \\
\\
\displaystyle
B_1 \frac{n_1\,n_2!}{(n_2-k)!}\mu_1\mu_2^{k} \prod_{j=n-k}^{n-1} \frac{1}{\mu_1n_1+\mu_2(j-n_1)-\lambda_2}\;
\frac{1}{\mu_1n_1+\mu_2n_2-\lambda}, &
\begin{array}{l}
k=0,\ldots,n_2,\;\\ i_1=i_2=0.
\end{array}
\end{array}\right.
\end{equation}
where $B_1$ is a normalizing constant.
\end{theorem}

\subsection{The Distribution of $(I_1(\frak{s}),I_2(\frak{s}))$ Given $K(\frak{s})$}

In this section we obtain the asymptotic distribution of $(I_1(\frak{s}),I_2(\frak{s}))$ conditional on $K(\frak{s})=k$, as $n\to\infty$.  We first show that  as $n\to\infty$, the probability of no idle servers of type $s_1$ goes to zero, and so the probability that customers need not wait goes to 1.
 Next we condition on $K(\frak{s})=k$ and show  $I_1(\frak{s})/n \stackrel{p}{\longrightarrow} f_1,\,I_2(\frak{s})/n \stackrel{p}{\longrightarrow} f_2$, where
 \[
 f_1 = \frac{m_1}{n}= \frac{T \theta}{T+1/\mu_1}, \qquad
 f_2 = \frac{m_2}{n}= \frac{T (1-\theta)}{T+1/\mu_2},
  \]
 where $T$ is given in (\ref{eqn.Tsolution}).
 Finally, we condition on $K(\frak{s})=k$ and show that  the scaled and centered values of $(I_1(\frak{s}),I_2(\frak{s}))$ converge in distribution to a bivariate normal distribution.

\begin{theorem}
\label{prop.ZeroQueue}
When $n\to\infty$, as long as $\rho<1$, $\delta<1$,
\[
P(I_1(\frak{s})=0) \to 0.
\]
\end{theorem}

From this theorem we see that when $n\to\infty$, $P(I_1(\frak{s})>0)\to 1$. Therefore, $P(K(\frak{s})=k,I_1(\frak{s})>0) \to P(K(\frak{s})=k)$ for any $0\leq k\leq I_2(\frak{s})$.
From equation (\ref{eqn.piKI}), given $K(\frak{s})=k$, the limiting stationary distribution as $n\to\infty$ is
\begin{eqnarray*}
&&P(I_1(\frak{s})=i_1, I_2(\frak{s})\to i_2|K(\frak{s})=k)\to P(I_1(\frak{s})=i_1, I_2(\frak{s})=i_2|K(\frak{s})=k,I_1(\frak{s})>0) \\
&& = B_1 {n_1\choose i_1}{n_2\choose i_2}i_1(i_1+i_2-k)! \frac{i_2!}{(i_2-k)!} \mu_1^{i_1}\mu_2^{i_2}\lambda^{-i_1-i_2-k}\lambda_1^{-k} \frac{1}{P(K(\frak{s})=k)} .
\end{eqnarray*}

\begin{theorem}
\label{thm.FluidLimit}
Conditional on $K(\frak{s})=k$, $\left(\frac{I_1(\frak{s})}{n},\,\frac{I_2(\frak{s})}{n}\right)$  converge to $(f_1,\,f_2)$ in probability for any $k\geq 0$. That is, for any $\epsilon>0$, when $n\to\infty$, we have
\begin{eqnarray*}
&&P\left(|I_1(\frak{s})-m_1|\geq \epsilon n \mbox{ or } |I_2(\frak{s})-m_2|\geq \epsilon n |K(\frak{s})=k\right)\to 0.
\end{eqnarray*}
\end{theorem}

After showing the fluid limit result,  we are now ready to show the central limit result.
\begin{theorem}
\label{thm.IdleDistribution}
For any $k\ge 0$, when $n\to\infty$, we have
\begin{equation}
\label{eqn.clt}
 \left(\left.\frac{I_1(\frak{s})- m_1 }{\sqrt{n}} ,\frac{I_2(\frak{s})-m_2}{\sqrt{n}} \right| K(\frak{s} = k) \right)
\Rightarrow  N \left( 0, \left[ \begin{array}{ll}  \sigma_1^2 & \rho \sigma_1 \sigma_2 \\
\rho \sigma_1 \sigma_2 & \sigma_2^2  \end{array} \right]  \right)
\end{equation}
where
\[
\rho = \left(\frac{(\theta-f_1)(1-\theta-f_2)f_1f_2}{(\theta f_2+f_1^2)((1-\theta)f_1+f_2^2)}\right)^{\frac{1}{2}},
\]
\[
\sigma_1 = \left(\frac{(\theta-f_1)f_1((1-\theta)f_1+f_2^2)}{\theta f_2^2+(1-\theta)f_1^2}\right)^{\frac{1}{2}},
\]
\[
\sigma_2 = \left(\frac{(1-\theta-f_2)f_2(\theta f_2+f_1^2)}{\theta f_2^2+(1-\theta)f_1^2}\right)^{\frac{1}{2}}.
\]
\end{theorem}

\subsection{The Distribution of $K(\frak{s})$, the Location of the First Type $s_1$ Server.}

\begin{theorem}
\label{thm.Kdist}
For any $k\geq 0$, as $n\to\infty$,
\begin{equation} \label{eqn.limK}
P(K(\frak{s})= k) \to \left(1-\frac{1-\beta}{\alpha}\right)\left(\frac{1-\beta}{\alpha}\right)^{k}.
\end{equation}
\end{theorem}

\subsection{Summary of stationary distribution}

Theorem~\ref{thm.Kdist} shows that $K(\frak{s})$ converges in distribution to a geometric distribution, so $P(K(\frak{s})<\infty)=1$. Therefore, we can extend Theorem~\ref{thm.FluidLimit} and Theorem~\ref{thm.IdleDistribution} into unconditional versions.


\begin{theorem} \label{thm.Uncondition}
When $n\to\infty$, $K(\frak{s})$ becomes independent of $I_1(\frak{s})$ and $I_2(\frak{s})$. $\left(\frac{I_1(\frak{s})-m_1}{\sqrt{n}},\,\frac{I_2(\frak{s})-m_2}{\sqrt{n}}\right)$ converges in distribution to the bivariate Normal distribution described in \ref{eqn.clt}. 
\end{theorem}

Consider a special case when $\mu_1=\mu_2=\mu$, we have $\theta=\beta$. $m_1$ and $m_2$ can be easily solved:
\[
m_1 = (1-\rho)n_1, \quad m_2 = (1-\rho)n_2.
\]
When $n\to\infty$, $\left(\frac{I_1(\frak{s})-(1-\rho)n_1}{\sqrt{n}},\,\frac{I_2(\frak{s})-(1-\rho)n_2}{\sqrt{n}}\right)$ converges in distribution to a bivariate Normal distribution with mean $(0,0)$,  variance
\[
\left(\,\rho \theta(1-\rho(1-\theta)),\, \rho (1-\theta)(1-\rho\theta)\,\right),
\]
and correlation
\[
\frac{\rho\sqrt{\theta(1-\theta)}}{\sqrt{(1-\rho(1-\theta))(1-\rho\theta)}}.
\]
The total idleness has mean of $(1-\rho)n$ and variance of
\[
Var(I_1(\frak{s}))+Var(I_2(\frak{s}))+2 Cov(I_1(\frak{s}),I_2(\frak{s})) = \rho n.
\]

\subsection{Comparison to the bipartite FCFS infinite matching model}
In the infinite matching model corresponding to the N-System there is an infinite sequence of customers, of types $c_1,c_2$, where the customer types are i.i.d.,  type is  $c_1$ with probability $\alpha$ and $c_2$ with probability $1-\alpha$, and an independent sequence of servers, of types  $s_1,s_2$, where the server types are i.i.d.,  type is  $s_1$ with probability $\beta$ and $s_2$ with probability $1-\beta$, and compatibility graph with arcs $\{(c_1,s_1),(c_1,s_2),(c_2,s_1)\}$.  Successive customers and servers are matched according to FCFS:  each server is matched to the first compatible customer that was not matched to a previous server, and each customer is matched to the first compatible server that was not matched to a previous customer.  

After $n$ of the customers have been matched, consider the sequence of remaining servers.  Let $K_n$ be the number of servers of type $s_2$ that are first in this sequence, preceding the first server of type $s_1$.  The $(K_n)_{n=0}^\infty$ is a Markov chain.
The steady state distribution for this Markov chain is that $P(K_\infty = k) = \left(1-\frac{1-\beta}{\alpha}\right)\left(\frac{1-\beta}{\alpha}\right)^{k},\,k\ge 0$, which is exactly the limiting distribution of $K(\fs)$ in (\ref{thm.Kdist}).

\section{Numerical Examples}
\label{sec.numerics}

We test our results by investigating an N-system. $\lambda=100$, $n_1=n_2 = 100$, $\mu_1 = \mu_2 = 1$, $\rho=0.5$.
From our approximation results, as long as $\alpha+\theta>1$, or $\alpha>0.5$, both pools should have similar utilization. So the average number of idle servers in each pool is close to 50, with variance of
$\theta\rho(1-\rho+\theta\rho)n=(1-\theta)\rho(1-\rho+(1-\theta)\rho)n=37.5$.
We use exact stationary distribution to verify this.

Now we can calculate the expectation and variance of idle number in each pool exactly, listed in the following table.
\begin{table}[h]
\centering \caption{The exact calculation} \label{table.1}
\begin{tabular}{ccccc}
\toprule
$\alpha$ & $E[I_1]$ & $Var[I_1]$ & $E[I_2]$ & $Var[I_2]$ \\
\midrule
0.8 & 49.8383 & 37.7049 & 50.1617 & 37.3814 \\
0.7 & 49.6482 & 38.078  & 50.3518 & 37.3743 \\
0.6 & 49.1787 & 39.2148 & 50.8213 & 37.5722 \\
0.55 &48.6055 & 40.8706 & 51.3945 & 38.0816 \\
0.5 & 47.333  & 44.883  & 52.667  & 39.549  \\
0.4 & 39.981  & 59.821  & 60.019  & 39.7854 \\
\bottomrule
\end{tabular}
\end{table}
We can see that when $\alpha>0.5$, the approximation is very good.
When $\alpha<0.5$, the approximation does not work. In fact, when $\alpha+\theta<1$ and system is large, complete resource pooling disappears, and server pool 1 seldom serves type s1 customers.
The N-system operates like 2 separate queues: pool 1 serves type s2 and pool 2 serves type s1.
The utilization of pool 1 is $\frac{(1-\alpha)\lambda}{n_1}$ and the utilization of pool 2 is $\frac{\alpha\lambda}{n_2}$.
From previous results, the number of idles servers in pool 1 can be approximated by a Normal distribution with mean $n_1-(1-\alpha)\lambda=40$ and variance $(1-\alpha)\lambda=60$;
whereas the number of idles servers in pool 2 can be approximated by a Normal distribution with mean $n_2-\alpha\lambda=60$ and variance $\alpha\lambda=40$.

When $\alpha$ goes down to 0.5, the approximation is getting worse. Note that $E[K(\frak{s})]\approx \frac{\alpha}{\alpha+\theta-1}$ is large when $\alpha$ is close to 0.5, making it not negligible.
We have a better approximation for $\alpha$ close to 0.5. 
When $K(\frak{s})=k$, we use $\theta^{'}(k)=\frac{n_1-1}{n-1-k}$ instead of $\theta$. The approximated average number of busy servers in pool 1 is $\theta^{'}(k)\rho n$.
\[
E[X]=E[E[X|K(\frak{s})]]=E[\theta^{'}(K)(1-\rho)n]=\sum_{k=0}{n_2}\frac{n_1-1}{n-1-k}\left(\frac{1-\theta}{\alpha}\right)^k\frac{\alpha+\theta-1}{\alpha} \rho n
\]
To make it consistent, we need $\theta = \frac{E[X]}{\rho n}$. Solving this equation gives $\theta$, and $E[X]$. In this example, when $\alpha = 0.6$,
the solution to
\[
\sum_{k=0}{n_2}\frac{n_1-1}{n-1-k}\left(\frac{1-\theta}{\alpha}\right)^k\frac{\alpha+\theta-1}{\alpha} = \theta
\]
is $\theta = 0.5093$. $E[I_1]=100-\theta \rho n=49.07$, which is closer to the true value 49.18.
The following table lists the comparison of the improved approximation and the true value:
\begin{table}[h]
\centering \caption{The improved approximation} \label{table.1}
\begin{tabular}{cccccc}
\toprule
$\alpha$ & 0.8  & 0.7 & 0.6 & 0.55 & 0.5   \\
\midrule
$E[I_1]$ & 49.8383 & 49.6482 & 49.1787 & 48.6055  & 47.333 \\
Approx.  & 49.83   & 49.62   & 49.07   & 48.29 & 46.46 \\
\bottomrule
\end{tabular}
\end{table}
We can see that even when $\alpha =0.5$, the improved approximation is not bad.

\section*{Acknowledgment}

We are grateful to Ivo Adan for helpful discussion of this paper.

\appendix

\section{Appendix:  Proofs for Section \ref{sec.stationary}}

\subsection{Proof of Lemma \ref{thm.sum} and Theorem \ref{thm.piI1I2K} }
\begin{proof}[Proof of Lemma \ref{thm.sum}]
We prove this lemma by induction. Define the left-hand-side as $C_m$. Step 2:
\[
C_2 = \frac{1}{a_1(a_1+a_2)} + \frac{1}{a_2(a_1+a_2)} = \frac{a_1+a_2}{a_1 a_2(a_1+a_2)} = \frac{1}{a_1 a_2}
\]
Step $m$:
\begin{eqnarray*}
C_m &=& \sum_{(A_1,\ldots,A_m)\in\mathcal{P}(a_1,\ldots,a_m) } \prod_{l=1}^m \left( \sum_{j=1}^l  A_j \right)^{-1} \\
&=& \frac{1}{\sum_{l=1}^m a_l}  \sum_{p=1}^m
\sum_{(A_1,\ldots,A_{m-1})\in \mathcal{P}(a_j: j\ne p) } \prod_{l=1}^{m-1} \left( \sum_{j=1}^l  A_j \right)^{-1} \\
&=& \frac{1}{\sum_{l=1}^m a_l}  \sum_{p=1}^m     \left( \prod_{j \ne p}  a_j \right)^{-1}          \\
&=& \frac{1}{\sum_{l=1}^m a_l}  \frac{\sum_{p=1}^m a_p}{\prod_{j=1}^m  a_j}  \\
&=& \left( \prod_{l=1}^m  a_l \right)^{-1}
\end{eqnarray*}
\end{proof}

\begin{proof}[Proof of Theorm \ref{thm.piI1I2K}]
Summation over the geometric terms $q_j=0,\ldots,\infty$ in (\ref{eqn.pistate}) gives
\begin{equation*}
\sum_{q_1,\cdots,q_{n-i}} \pi(\frak{s}) = \left\{ \begin{array}{ll}
\displaystyle
B \prod_{l=1}^{n-i_1-i_2} \left( \sum_{j=1}^l  \mu(S_j) \right)^{-1} \left(\frac{1}{\lambda}\right)^{i_1+i_2-k} \left(\frac{1}{\lambda_1}\right)^{k}, &
\begin{array}{l}
k=0,\ldots,n_2,\; \\ i_1=0,\ldots,n_1\;\\ i_2=k,\ldots,n_2,
\end{array} \\
\\
\displaystyle
B \prod_{l=1}^{n-k-1} \left( \sum_{j=1}^l  \mu(S_j) \right)^{-1}\prod_{j=n-k}^{n-i_2} \frac{1}{\mu_1n_1+\mu_2(j-n_1)-\lambda_2}\;
\left(\frac{1}{\lambda_1}\right)^{i_2}, \quad &
\begin{array}{l}
k=1,\ldots,n_2,\; \\ i_1=0,\;\\ i_2=1,\ldots,k,
\end{array} \\
\\
\displaystyle
B \prod_{l=1}^{n-k-1} \left( \sum_{j=1}^l  \mu(S_j) \right)^{-1}\prod_{j=n-k}^{n-1} \frac{1}{\mu_1n_1+\mu_2(j-n_1)-\lambda_2}\;
\frac{1}{\mu_1n_1+\mu_2n_2-\lambda}, &
\begin{array}{l}
k=0,\ldots,n_2,\;\\ i_1=i_2=0.
\end{array}
\end{array}\right.
\end{equation*}
Next we see that in this expression, permutations of $S_1,\ldots,S_n$ with the same $(k,i_1,i_2)$ have a similar structure.   We now  sum over all the permutations of the appropriate $S_j,\,1\le j \le n-\max\{k+1,i_1+i_2\}$.  By  Lemma~\ref{thm.sum} we  obtain
\begin{equation}\label{eqn.pipermutation}
\left\{ \begin{array}{ll}
\displaystyle
B \mu_1^{i_1-n_1}\mu_2^{i_2-n_2} \left(\frac{1}{\lambda}\right)^{i_1+i_2-k} \left(\frac{1}{\lambda_1}\right)^{k}, &
\begin{array}{l}
k=0,\ldots,n_2,\;  \\
i_1=1,\ldots,n_1,\; \\ i_2=k,\ldots,n_2,
\end{array} \\
\\
\displaystyle
B \mu_1^{1-n_1}\mu_2^{k-n_2} \prod_{j=n-k}^{n-i_2} \frac{1}{\mu_1n_1+\mu_2(j-n_1)-\lambda_2}\;
\left(\frac{1}{\lambda_1}\right)^{i_2}, \quad &
\begin{array}{l}
k=1,\ldots,n_2,\; \\ i_1=0,\;\\ i_2=1,\ldots,k,
\end{array} \\
\\
\displaystyle
B \mu_1^{1-n_1}\mu_2^{k-n_2} \prod_{j=n-k}^{n-1} \frac{1}{\mu_1n_1+\mu_2(j-n_1)-\lambda_2}\;
\frac{1}{\mu_1n_1+\mu_2n_2-\lambda}, &
\begin{array}{l}
k=0,\ldots,n_2,\;\\ i_1=i_2=0.
\end{array}
\end{array}\right.
\end{equation}
Each permutation of the remaining servers, $S_j,\, n-\max\{k+1,i_1+i_2\} < j \le n$ has the same stationary probability.  It remains to count the number of permutations.
When $i_1=0$ we have $i_2 \le k$.  For each permutation we choose 1 type $s_1$ server and $k$ out of $n_2$ type $s_2$ servers to form the last $k+1$ servers. The number of permutations is
\[
n_1{n_2\choose k}k! = \frac{n_1\,n_2!}{(n_2-k)!}.
\]
When $i_1>0$, we have $i_2\geq k$. For each permutation, we choose $i_1$ out of $n_1$ type $s_1$ servers and $i_2$ out of $n_2$ type $s_2$ servers. We then choose 1 from the $i_1$ idle  servers of type $s_1$, and $k$ from the $i_2$ idle servers of type $s_2$ to obtain the last $k+1$ servers. The number of permutations is
\[
{n_1 \choose i_1}{n_2\choose i_2}i_1{i_2 \choose k}(i_1+i_2-k-1)!k! ={n_1\choose i_1}{n_2\choose i_2} \frac{i_1 i_2! (i_1+i_2-k-1)!}{(i_2-k)!}.
\]
Multiplying the terms in (\ref{eqn.pipermutation}) by the appropriate number of permutations and defining $B_1=B\mu_1^{-n_1}\mu_2^{-n_2}$ gives (\ref{eqn.piKI}).
\end{proof}

\subsection{Proofs of Theorems \ref{prop.ZeroQueue}, \ref{thm.FluidLimit} and \ref{thm.IdleDistribution}}

\begin{proof}[Proof of Theorem \ref{prop.ZeroQueue}]
We prove the theorem in three steps:
\begin{description}
\item[(i)]
We show that
\[
P\left(I_1(\frak{s})=0\right) \sim B_1 \frac{1}{1-\delta} \times \left\{ \begin{array}{ll}
\sqrt{2\pi n_2}\exp\left(n_2\left(-\log\,\kappa+\kappa-1\right)\right)  & 0<  \kappa < 1 \\
\sqrt{2\pi n_2}/2 &  \kappa =1  \\
B_1 \frac{1}{1-\delta}\left(\frac{1-(1-\alpha)\rho}{1-\rho}+\frac{1}{\kappa-1}\right)   &  \kappa > 1
\end{array}   \right.
\]
where $\kappa:=\frac{\lambda_1}{\mu_2n_2}$. Note that  by $\alpha + \beta >1$ we have $\kappa > \frac{1}{1 + \mu_2 T}$. Note also that $-\log\,\kappa+\kappa-1\geq 0$.
\item[(ii)]
We show that
\begin{eqnarray*}
&&P(I_1(\frak{s})=\lceil m_1\rceil, I_2(\frak{s})=\lceil m_2\rceil,K(\frak{s})=0)  \\
&&\sim B_1\left(\frac{2\pi\beta n_1n_2}{(n_1-m_1)(n_2-m_2)m_2}\right)^{1/2}
\exp\left[-n_1 \left(\log\left(1-\frac{m_1}{n_1}\right)+\frac{m_1}{n_1} \right)\right]  \\
&& \qquad \exp\left[-n_2 \left(\log\left(1-\frac{m_2}{n_2}\right)+\frac{m_2}{n_2} \right)\right]
\end{eqnarray*}
where $m_1$ and $m_2$ are defined in (\ref{eqn.MeanIdle}).
\item[(iii)]
We show that as $n\to\infty$
\[
\frac{P\left(I_1(\frak{s})=0\right) }
{P(I_1(\frak{s})=\lceil m_1\rceil, I_2(\frak{s})=\lceil m_2\rceil,K(\frak{s})=0)}  \to 0
\]
which proves the proposition.
\end{description}
The details of the proofs of these three steps are as follows:

\textbf{Proof of (i):}

First we calculate
\begin{eqnarray*}
&& P(I_1(\frak{s}) = 0,\,I_2(\frak{s})=0)=\sum_{k=0}^{n_2}\pi(k,0,0) \\
&&  \qquad  = \sum_{k=0}^{n_2}B_1 \frac{n_1\,n_2!}{(n_2-k)!}\mu_1\mu_2^{k} \prod_{j=n-k}^{n-1} \frac{1}{\mu_1n_1+\mu_2(j-n_1)-\lambda_2}\;\frac{1}{\mu_1n_1+\mu_2n_2-\lambda}.
\end{eqnarray*}
We use induction to calculate
\[
U_m:=\sum_{k=m}^{n_2}\frac{\mu_2^k}{(n_2-k)!} \prod_{j=n-k}^{n-1} \frac{1}{\mu_1n_1+\mu_2(j-n_1)-\lambda_2}
\]
from $m=n_2$ to $m=1$.
When $m=n_2$,
\[
U_{n_2}=\frac{\mu_2^{n_2}}{(n_2-n_2)!}\frac{1}{\mu_1n_1-\lambda_2} \prod_{j=n-n_2+1}^{n-1} \frac{1}{\mu_1n_1+\mu_2(j-n_1)-\lambda_2}.
\]
Suppose
\[
U_{m+1} = \frac{\mu_2^{m+1}}{(n_2-m-1)!}\frac{1}{\mu_1 n_1-\lambda_2}\prod_{j=n-m}^{n-1} \frac{1}{\mu_1n_1+\mu_2(j-n_1)-\lambda_2}
\]
then
\begin{eqnarray*}
&U_m&=  \frac{\mu_2^{m+1}}{(n_2-m-1)!} \frac{1}{\mu_1n_1-\lambda_2}\prod_{j=n-m}^{n-1}\frac{1}{\mu_1n_1+\mu_2(j-n_1)-\lambda_2} \\
&&\qquad + \frac{\mu_2^{m}}{(n_2-m)!}\prod_{j=n-m}^{n-1} \frac{1}{\mu_1n_1+\mu_2(j-n_1)-\lambda_2}\\
&&= \frac{\mu_2^{m}}{(n_2-m)!}\prod_{j=n-m}^{n-1} \frac{1}{\mu_1n_1+\mu_2(j-n_1)-\lambda_2}\left(\frac{\mu_2(n_2-m)}{\mu_1 n_1-\lambda_2}+1\right)\\
&&= \frac{\mu_2^{m}}{(n_2-m)!}\frac{1}{\mu_1 n_1-\lambda_2}\prod_{j=n+1-m}^{n-1}\frac{1}{\mu_1n_1+\mu_2(j-n_1)-\lambda_2}.
\end{eqnarray*}
Therefore, the induction is valid and we have
\[
U_1=\frac{\mu_2}{(n_2-1)!} \frac{1}{\mu_1 n_1-\lambda_2}.
\]
\begin{eqnarray*}
&P(I_1(\frak{s})=0,\,I_2(\frak{s})=0)&= U_1 B_1n_1n_2!\frac{\mu_1}{\mu_1n_1+\mu_2n_2-\lambda}+\pi(0,0,0)\\
&&=B_1  \frac{\mu_2 n_2}{\mu_1 n_1-\lambda_2}\;\frac{\mu_1n_1}{\mu_1n_1+\mu_2n_2-\lambda}+B_1\frac{\mu_1 n_1}{\mu_1n_1+\mu_2n_2-\lambda} \\
&&=B_1 \frac{\mu_1n_1}{\mu_1 n_1-\lambda_2}\;\frac{\mu_1n_1+\mu_2n_2-\lambda_2}{\mu_1n_1+\mu_2n_2-\lambda}\\
&&=B_1 \frac{1}{1-\delta} \;\frac{1-(1-\alpha)\rho}{1-\rho}.
\end{eqnarray*}

Next we calculate
\[
P(I_1(\frak{s})=0,\,I_2(\frak{s})>0)=\sum_{k=1}^{n_2}\sum_{i_2=1}^{k}\pi(k,0,i_2)=\sum_{i_2=1}^{n_2}\sum_{k=i_2}^{n_2}\pi(k,0,i_2).
\]
Similar to the induction calculating $U_m$ above,  we can obtain
\begin{eqnarray*}
&\sum_{k=i_2}^{n_2}\pi(k,0,i_2)&=B_1\left(\frac{1}{\lambda_1}\right)^{i_2}n_1 \mu_1 n_2!\sum_{k=i_2}^{n_2}\frac{\mu_2^k}{(n_2-k)!} \prod_{j=n-k}^{n-i_2} \frac{1}{\mu_1n_1+\mu_2(j-n_1)-\lambda_2} \\
&& = B_1 \left(\frac{1}{\lambda_1}\right)^{i_2} n_1\mu_1\,n_2! \frac{\mu_2^{i_2}}{(n_2-i_2)!}\frac{1}{\mu_1 n_1-\lambda_2}
\\&& = B_1 \frac{1}{1-\delta}  \left( \frac{\mu_2}{\lambda_1}\right )^{i_2} \frac{n_2!}{(n_2-i_2)!}
\end{eqnarray*}
Therefore,
\begin{eqnarray*}
&P(I_1(\frak{s})=0,\,I_2(\frak{s})>0) &= B_1\frac{1}{1-\delta}n_2!\sum_{i_2=1}^{n_2}\left(\frac{\lambda_1}{\mu_2}\right)^{-i_2} \frac{1}{(n_2-i_2)!}\\
&&= B_1 \frac{1}{1-\delta}n_2!\left(\frac{\lambda_1}{\mu_2}\right)^{-n_2} \sum_{i_1^{'}=0}^{n_2-1}\left(\frac{\lambda_1}{\mu_2}\right)^{i_2^{'}} \frac{1}{i_2^{'}!} \\
&&= B_1 \frac{1}{1-\delta} n_2!  \left(\frac{\mu_2}{\lambda_1}\right)^{n_2}
\exp\left(\frac{\lambda_1}{\mu_2} \right) P(X<n_2) \\
&& = B_1 \frac{1}{1-\delta} \frac{P(X<n_2)}{P(X=n_2)},
\end{eqnarray*}
where $X$ is a Poisson random variable with parameter $\frac{\lambda_1}{\mu_2}$.
Using Stirling's approximation,
\begin{eqnarray*}
&P(X=n_2) & = \frac{1}{n_2!}\left(\frac{\lambda_1}{\mu_2}\right)^{n_2}\exp\left(-\frac{\lambda_1}{\mu_2}\right) \\
&&\sim \frac{1}{\sqrt{2\pi n_2}} \left(\frac{\lambda_1}{\mu_2n_2}\right)^{n_2}\exp\left(n_2-\frac{\lambda_1}{\mu_2}\right)\\
&&=\frac{1}{\sqrt{2\pi n_2}} \exp\left(n_2\left(\log\left(\frac{\lambda_1}{\mu_2n_2}\right)+1-\frac{\lambda_1}{\mu_2n_2}\right)\right)\\
&&=\frac{1}{\sqrt{2\pi n_2}}\exp\left(n_2\left(\log\kappa +1 -\kappa\right)\right),
\end{eqnarray*}
Recall that $\kappa=\frac{\lambda_1}{\mu_2n_2}$ and note that $\log\kappa+1-\kappa\leq 0$. Note also that when $n\to\infty$, $X$ can be approximated by a Normal distribution with mean $\frac{\lambda_1}{\mu_2}$ and variance $\frac{\lambda_1}{\mu_2}$. Next we analyze $\frac{P(X<n_2)}{P(X=n_2)}$ in 3 cases depending on $\kappa$.
\begin{itemize}
\item When $0<\kappa<1$, from the Normal distribution approximation, when $n\to \infty$, $P(X<n_2)\to 1$.
Therefore,
\[
P\left(I_1(\frak{s})=0,\,I_2(\frak{s})>0\right)\sim B_1 \frac{1}{1-\delta}
\left(\sqrt{2\pi n_2}\exp\left(-n_2\left(\log\kappa+1-\kappa\right)\right)\right).
\]
\item When $\kappa=1$, $-\log\kappa+\kappa-1=0$. When $n\to \infty$, the Normal distribution approximation gives $P(X<n_2)\to \frac{1}{2}$.
\[
P\left(I_1(\frak{s})=0,\,I_2(\frak{s})>0\right)\sim B_1 \frac{1}{1-\delta}\frac{1}{2}\sqrt{2\pi n_2}.
\]
\item When $\kappa> 1$, when $n\to\infty$, the Normal distribution approximation gives $P(X<n_2)\to 0$. We need more care to treat this case. For any $1\leq j\leq n_2$,
\[
\frac{P(X=n_2-j)}{P(X=n_2)} = \frac{\left(\frac{\lambda_1}{\mu_2}\right)^{n_2-j}\frac{1}{(n_2-j)!}}{\left(\frac{\lambda_1}{\mu_2}\right)^{n_2}\frac{1}{n_2!}}
=\frac{n_2!}{\kappa^j n_2^j (n_2-j)!} < \frac{1}{\kappa^j}.
\]
Therefore,
\[
\frac{P(X<n_2)}{P(X=n_2)}\leq \sum_{j=1}^{n_2}\frac{1}{\kappa^j}<\frac{1}{\kappa-1}.
\]
In fact, for any fixed $j$, when $n\to\infty$,
\[
\frac{P(X=n_2-j)}{P(X=n_2)} \to \frac{1}{\kappa^j}.
\]
For any $\epsilon>0$, let $J:=\lceil\frac{-\log\epsilon}{\log\kappa}\rceil$. We have $\epsilon\geq \kappa^{-J}$.
There exists an $N$ such that when $n>N$, for any $1\leq j\leq J$,
\[
\frac{P(X=n_2-j)}{P(X=n_2)} - \frac{1}{\kappa^j}>-\frac{\epsilon}{J}.
\]
Therefore,
\[
\frac{P(X<n_2)}{P(X=n_2)} > \sum_{j=1}^{J}\frac{1}{\kappa^j}-\epsilon = \frac{1-\kappa^{-J}}{\kappa-1}-\epsilon\geq\frac{1}{\kappa-1}-\frac{\kappa \epsilon}{\kappa-1}.
\]
Therefore, when $n\to\infty$,
\[
\frac{P(X<n_2)}{P(X=n_2)} \to \frac{1}{\kappa-1}
\]
We have
\begin{equation*}
P\left(I_1(\frak{s})=0,\, I_2(\frak{s})>0\right) \sim B_1 \frac{1}{1-\delta}\frac{1}{\kappa-1}.
\end{equation*}
\end{itemize}

In summary, when $\kappa\leq 1$, $P\left(I_1(\frak{s})=0,\,I_2(\frak{s})=0\right)$ is negligible compared with $P\left(I_1(\frak{s})=0,\,I_2(\frak{s})>0\right)$ when $n\to\infty$. We have
\[
P\left(I_1(\frak{s})=0\right) \sim B_1 \frac{1}{1-\delta} \times \left\{ \begin{array}{ll}
\sqrt{2\pi n_2}\exp\left(n_2\left(-\log\,\kappa+\kappa-1\right)\right)  & 0<  \kappa < 1 \\
\sqrt{2\pi n_2}/2 &  \kappa =1  \\\frac{1-(1-\alpha)\rho}{1-\rho}+\frac{1}{\kappa-1}& \kappa > 1
\end{array}   \right.
\]

\textbf{Proof of (ii):}

From equation (\ref{eqn.piKI}) we have
\begin{eqnarray*}
&&  P(I_1(\frak{s})=\lceil m_1\rceil, I_2(\frak{s})=\lceil m_2\rceil,K(\frak{s})=0) \\
&&  = B_1{n_1\choose \lceil m_1\rceil}{n_2\choose \lceil m_2\rceil} \lceil m_1\rceil (\lceil m_1\rceil+\lceil m_2\rceil-1)!\mu_1^{\lceil m_1\rceil}\mu_2^{\lceil m_2\rceil} \left(\frac{1}{\lambda}\right)^{\lceil m_1\rceil+\lceil m_2\rceil} \\
&& > \frac{B_1}{ m_1^2 m_2 (m_1+m_2)^2 \mu_1 \mu_2 } {n_1\choose m_1}{n_2\choose  m_2}  m_1 ( m_1+ m_2-1)!\mu_1^{ m_1}\mu_2^{ m_2} \left(\frac{1}{\lambda}\right)^{ m_1+ m_2} \\
&& > \frac{B_1}{ n^5 \mu_1 \mu_2 } {n_1\choose m_1}{n_2\choose  m_2}  m_1 ( m_1+ m_2-1)!\mu_1^{ m_1}\mu_2^{ m_2} \left(\frac{1}{\lambda}\right)^{ m_1+ m_2} \\
&&\sim \frac{B_1}{ n^5 \mu_1 \mu_2 }\frac{m_1}{m_1+m_2} \frac{n_1!n_2!}{(n_1-m_1)!m_1!(n_2-m_2)!m_2!} (m_1+m_2)!\mu_1^{m_1}\mu_2^{m_2}\lambda^{-m_1-m_2} \\
&&\sim \frac{B_1}{ n^5 \mu_1 \mu_2 }\left(\frac{2\pi m_1n_1n_2}{(m_1+m_2)(n_1-m_1)(n_2-m_2)m_2}\right)^{1/2} \frac{n_1^{n_1}n_2^{n_2}}{(n_1-m_1)^{n_1-m_1}m_1^{m_1}(n_2-m_2)^{n_2-m_2}m_2^{m_2}} \\
&&\qquad \times \left(\frac{m_1+m_2}{e}\right)^{m_1+m_2} \left(\frac{\mu_1}{\lambda}\right)^{m_1}\left(\frac{\mu_2}{\lambda}\right)^{m_2} \\
&&= \frac{B_1}{ n^5 \mu_1 \mu_2 }\left(\frac{2\pi m_1n_1n_2}{(m_1+m_2)(n_1-m_1)(n_2-m_2)m_2}\right)^{1/2} \\
&& \qquad \times \left(\frac{n_1}{n_1-m_1}\right)^{n_1}\left(\frac{n_2}{n_2-m_2}\right)^{n_2}
\left(\frac{m_1+m_2}{m_1}\right)^{m_1}\left(\frac{m_1+m_2}{m_2}\right)^{m_2} \\
&&\qquad \times \exp(-m_1-m_2)
\left(\frac{\mu_1(n_1-m_1)}{\lambda}\right)^{m_1}\left(\frac{\mu_2(n_2-m_2)}{\lambda}\right)^{m_2} \\
&&=\frac{B_1}{ n^5 \mu_1 \mu_2 }\left(\frac{2\pi\beta n_1n_2}{(n_1-m_1)(n_2-m_2)m_2}\right)^{1/2} \left(\frac{n_1}{n_1-m_1}\right)^{n_1}\left(\frac{n_2}{n_2-m_2}\right)^{n_2}\exp(-m_1-m_2)\\
&&=\frac{B_1}{ n^5 \mu_1 \mu_2 }\left(\frac{2\pi\beta n_1n_2}{(n_1-m_1)(n_2-m_2)m_2}\right)^{1/2} \exp\left(-n_1\left(\log\left(1-\frac{m_1}{n_1}\right)+\frac{m_1}{n_1}\right)\right) \\
&&\qquad \times  \exp\left(-n_2\left(\log\left(1-\frac{m_2}{n_2}\right)+\frac{m_2}{n_2}\right)\right).
\end{eqnarray*}
The second equality is due to $\frac{m_1}{m_1+m_2}=\beta$, $\frac{m_2}{m_1+m_2}=1-\beta$, $\frac{\mu_1(n_1-m_1)}{\lambda}=\beta$, $\frac{\mu_2(n_2-m_2)}{\lambda}=1-\beta$.

\textbf{Proof of (iii):}

Since  $\log(1-x)+x < 0$ when $0<x<1$, we have
\[
\log\left(1-\frac{m_1}{n_1}\right)+\frac{m_1}{n_1}<0 \mbox{ and } \log\left(1-\frac{m_2}{n_2}\right)+\frac{m_2}{n_2}<0.
\]
When $n\to\infty$, note that $\left(\frac{2\pi\beta n_1n_2}{(n_1-m_1)(n_2-m_2)m_2}\right)^{1/2}$ is of the order of $n^{-1/2}$. Therefore, $P(I_1(\frak{s})=\lceil m_1\rceil, I_2(\frak{s})=\lceil m_2\rceil,K(\frak{s})=0)/B_1$ increases exponentially. When $\kappa>1$, $P(I_1(\frak{s})=0)/B_1$ converges to a constant; when $\kappa=1$, $P(I_1(\frak{s})=0)/B_1$ increases in the order of $\sqrt{n}$.
Therefore, when $n\to\infty$ and $\kappa\geq 1$,
\[
\frac{P\left(I_1(\frak{s})=0\right)}{P(I_1(\frak{s})=\lceil m_1\rceil, I_2(\frak{s})=\lceil m_2\rceil,K(\frak{s})=0)}\to 0.
\]
When $\kappa<1$,
\begin{eqnarray*}
&&\frac{P\left(I_1(\frak{s})=0\right)}{P(I_1(\frak{s})=\lceil m_1\rceil, I_2(\frak{s})=\lceil m_2\rceil,K(\frak{s})=0)}
\sim \left(\frac{(n_1-m_1)(n_2-m_2)m_2}{\beta n_1(1-\delta)^2}\right)^{1/2}  \\
&&\qquad \times \exp\left(n_1\left(\log\left(1-\frac{m_1}{n_1}\right)+\frac{m_1}{n_1}\right)\right)
\exp\left(n_2\left(\log\left(1-\frac{m_2}{n_2}\right)+\frac{m_2}{n_2}-\log\kappa+\kappa-1\right)\right).
\end{eqnarray*}
We have that
\begin{eqnarray*}
&\log\left(1-\frac{m_2}{n_2}\right)+\frac{m_2}{n_2}-\log\kappa+\kappa-1
&=\log\left(\frac{n_2-m_2}{n_2}\frac{\mu_2n_2}{\lambda_1}\right)+\frac{\lambda_1-(n_2-m_2)\mu_2}{n_2\mu_2} \\
&&=\log\left(\frac{(n_2-m_2)\mu_2}{\lambda_1}\right)+\frac{\lambda_1-(n_2-m_2)\mu_2}{\lambda_1}\kappa \\
&&<\log\left(\frac{1-\beta}{\alpha}\right)+\frac{\alpha-(1-\beta)}{\alpha} \\
&&=\log\left(1-\frac{\alpha+\beta-1}{\alpha}\right)+\frac{\alpha+\beta-1}{\alpha} < 0
\end{eqnarray*}
Therefore, when $n\to\infty$,
\[
\frac{P\left(I_1(\frak{s})=0\right)}{P(I_1(\frak{s})=\lceil m_1\rceil, I_2(\frak{s})=\lceil m_2\rceil,K(\frak{s})=0)}\to 0.
\]
This completes the proof that when $n\to\infty$,
\[
P(I_1(\frak{s})=0)\to 0
\]
\end{proof}

\begin{proof}[Proof of Theorem \ref{thm.IdleDistribution}]
First we show that the weak convergence is valid given $K(\frak{s})=0$. Then we show that the same holds when $K(\frak{s})=k$, for any fixed $k$.
When $K(\frak{s})=0$, we prove the  convergence in probability in 2 steps.
\begin{itemize}
\item [(i)] We show that for all states $|I_1(\frak{s})-m_1|\geq \epsilon n \mbox{ or } |I_2(\frak{s})-m_2|\geq \epsilon n$, the conditional probability is dominated by a bounded constant multiple of the conditional probability of some point on the boundary of the rectangle $|I_1(\frak{s})-m_1| \leq \epsilon n$ $\times$ $|I_2(\frak{s})-m_2|\leq \epsilon n$.
\item [(ii)] When $n\to\infty$, we
approximate the conditional probability of the points in the rectangle $|I_1(\frak{s})-m_1| \leq \epsilon n$ $\times$ $|I_2(\frak{s})-m_2|\leq \epsilon n$. We then show that the probability of points on the boundary is negligible compared with the conditional probability at $(\lceil m_1\rceil,\lceil m_2\rceil)$.
\end{itemize}
\textbf{Proof of (i):}

\begin{eqnarray*}
&&P(I_1(\frak{s})=i_1, I_2(\frak{s})=i_2|K(\frak{s})=0) \\
&& = B_2{n_1\choose i_1}{n_2\choose i_2}i_1 (i_1+i_2-1)! \mu_1^{i_1}\mu_2^{i_2}\lambda^{-i_1-i_2} \\
&& = B_2 \frac{n_1!\,n_2!}{(n_1-i_1)!(i_1-1)!(n_2-i_2)!i_2!} (i_1+i_2-1)!\left(\frac{\mu_1}{\lambda}\right)^{i_1}\left(\frac{\mu_2}{\lambda}\right)^{i_2},
\end{eqnarray*}
where $B_2=B_1/P(K(\frak{s})=0)$.
\[
\frac{P(I_1(\frak{s})=i_1+1, I_2(\frak{s})=i_2|K(\frak{s})=0)}{P(I_1(\frak{s})=i_1, I_2(\frak{s})=i_2|K(\frak{s})=0)}
=\frac{(i_1+i_2)(n_1-i_1)\mu_1}{i_1\lambda}=\beta\frac{n_1-i_1}{n_1-m_1}\frac{i_1+i_2}{i_1}.
\]
\[
\frac{P(I_1(\frak{s})=i_1, I_2(\frak{s})=i_2+1|K(\frak{s})=0)}{P(I_1(\frak{s})=i_1, I_2(\frak{s})=i_2|K(\frak{s})=0)}
=\frac{(i_1+i_2)(n_2-i_2)\mu_2}{(i_2+1)\lambda}=(1-\beta)\frac{n_2-i_2}{n_2-m_2}\frac{i_1+i_2}{i_2+1}.
\]
We look at several cases:
\begin{itemize}
\item When $i_1\leq m_1$ and $(1-\beta)i_1<\beta i_2$, we have $\frac{i_1+i_2}{i_1}>\frac{1}{\beta}$. Therefore, $\frac{P(I_1(\frak{s})=i_1+1, I_2(\frak{s})=i_2|K(\frak{s})=0)} {P(I_1(\frak{s})=i_1,I_2(\frak{s})=i_2|K(\frak{s})=0)}>1$;
\item when $i_2\leq m_2$ and $(1-\beta)i_1>\beta i_2+1$, we have $\frac{i_1+i_2}{i_2+1} > \frac{1}{1-\beta}$. Therefore, $\frac{P(I_1(\frak{s})=i_1, I_2(\frak{s})=i_2+1|K(\frak{s})=0)} {P(I_1(\frak{s})=i_1,I_2(\frak{s})=i_2|K(\frak{s})=0)}>1$;
\item when $i_1> m_1$, $i_2>m_2$ and $(1-\beta)i_1\geq \beta i_2$, we have $\frac{i_1+i_2}{i_1}\leq \frac{1}{\beta}$. Therefore, $\frac{P(I_1(\frak{s})=i_1+1, I_2(\frak{s})=i_2|K(\frak{s})=0)}{P(I_1(\frak{s})=i_1, I_2(\frak{s})=i_2|K(\frak{s})=0)}<1$;
\item when $i_1> m_1$, $i_2>m_2$ and $(1-\beta)i_1\leq \beta i_2+1$, we have $\frac{i_1+i_2}{i_2+1}\leq \frac{1}{1-\beta}$. Therefore, $\frac{P(I_1(\frak{s})=i_1, I_2(\frak{s})=i_2+1|K(\frak{s})=0)}{P(I_1(\frak{s})=i_1, I_2(\frak{s})=i_2|K(\frak{s})=0)}<1$;
\item when $\beta i_2\leq (1-\beta)i_1\leq \beta i_2+1$, $i_1\leq m_1-\epsilon n$ and $i_2 \leq m_2-\epsilon n$. As long as $\frac{n_2-i_2}{n_2-m_2}\frac{i_2}{i_2+1}>1$, we have $\frac{P(I_1(\frak{s})=i_1, I_2(\frak{s})=i_2+1|K(\frak{s})=0)}{P(I_1(\frak{s})=i_1, I_2(\frak{s})=i_2|K(\frak{s})=0)}>1$. When $n$ is large, this requires
    \[ i_2> i_2^*:=\frac{1-\theta-f_2}{f_2}. \]
    As long as $\frac{n_1-i_1}{n_1-m_1}\frac{i_1-1}{i_1}>1$, we have $\frac{P(I_1(\frak{s})=i_1+1, I_2(\frak{s})=i_2|K(\frak{s})=0)}{P(I_1(\frak{s})=i_1, I_2(\frak{s})=i_2|K(\frak{s})=0)}>1$. When $n$ is large, this requires
    \[  i_1> i_1^*:=\frac{\theta}{f_1}. \]
\end{itemize}
For all $i_1>i_1^*$ or $i_2>i_2^*$, we can move the state to a neighbour state with larger steady state probability, shown as Figure~\ref{fig.Dominance}.
\begin{figure}[htb]
  \begin{center}
  \includegraphics[width=0.60\linewidth]{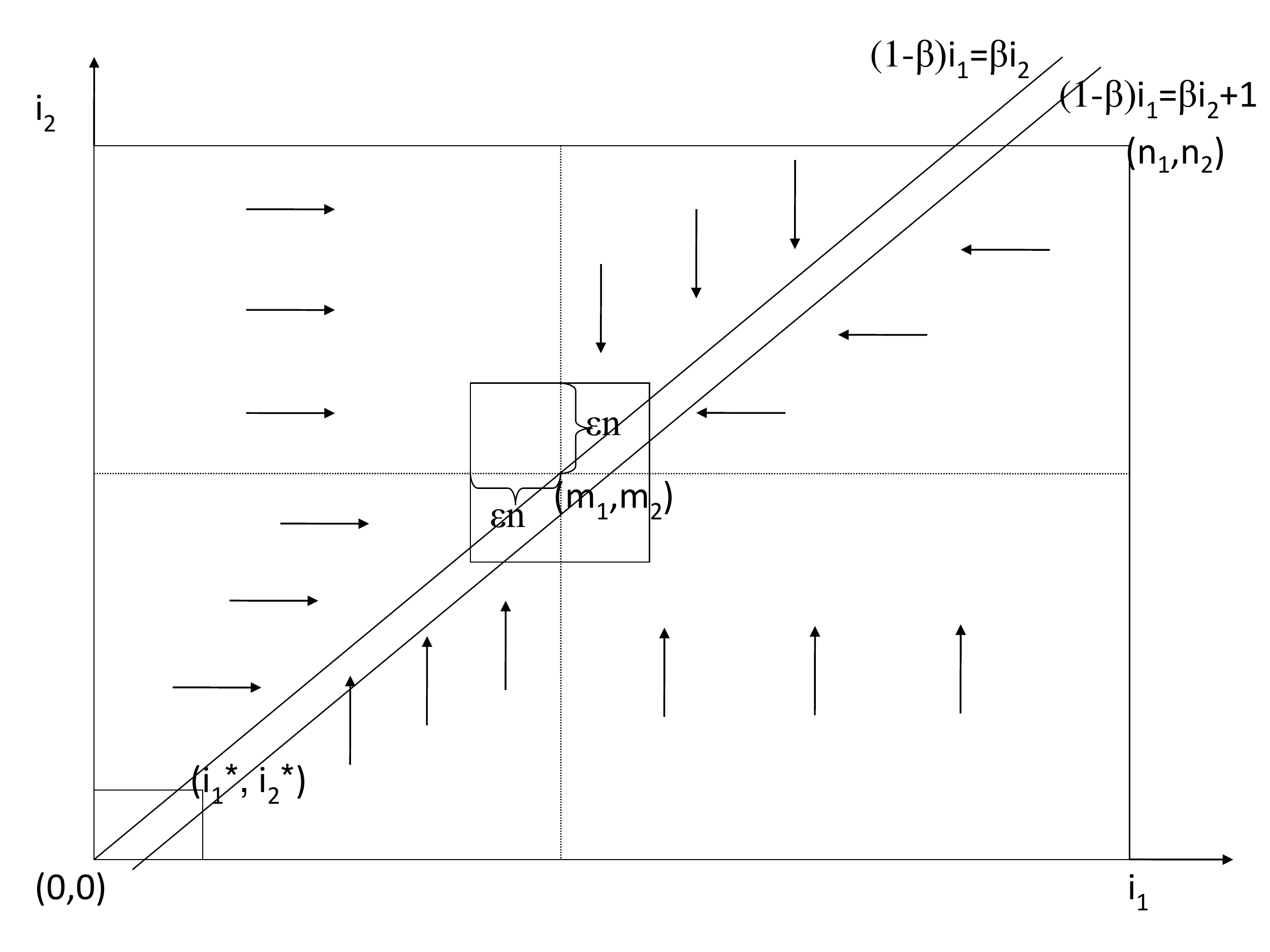}
  \end{center}
  \caption{The dominance of steady steady probability}
  \label{fig.Dominance}
\end{figure}

Eventually the movement stops at the boundary which are $\epsilon n$ away from $(m_1,m_2)$. Therefore, the probability of any state $(i_1,i_2)$ satisfying $i_1>i_1^*$ or $i_2>i_2^*$ would be dominated by the probability of some point at the boundary.

For any $(i_1,i_2)$ satisfying $i_1\leq i_1^*$ and $i_2\leq i_2^*$, since
\[
\frac{P(I_1(\frak{s})=i_1+1, I_2(\frak{s})=i_2|K(\frak{s})=0)}{P(I_1(\frak{s})=i_1, I_2(\frak{s})=i_2|K(\frak{s})=0)}
>\beta \mbox{ and }
\frac{P(I_1(\frak{s})=i_1, I_2(\frak{s})=i_2+1|K(\frak{s})=0)}{P(I_1(\frak{s})=i_1, I_2(\frak{s})=i_2|K(\frak{s})=0)}
>1-\beta,
\]
We have
\[
P(I_1(\frak{s})=i_1, I_2(\frak{s})=i_2|K(\frak{s})=0)< \frac{1}{\beta^{i_1^*+1}(1-\beta)^{i_2^*+1}}P(I_1(\frak{s})=i_1^*+1, I_2(\frak{s})=i_2^*+1|K(\frak{s})=0)
\]
and $P(I_1(\frak{s})=i_1^*+1, I_2(\frak{s})=i_2^*+1|K(\frak{s})=0)$ is dominated by the probability of some point at the boundary.

\textbf{Proof of (ii):}

When $i_1\in \left[m_1-\epsilon n,m_1+\epsilon n\right]$ and $i_2\in\left[m_2-\epsilon n,m_2+\epsilon n\right]$, and $n$ grows large, we can use Stirling's approximation.
\begin{eqnarray*}
&&P(I_1(\frak{s})=i_1, I_2(\frak{s})=i_2|K(\frak{s})=0) \\
&& = B_2{n_1\choose i_1}{n_2\choose i_2}i_1 (i_1+i_2-1)! \mu_1^{i_1}\mu_2^{i_2}\lambda^{-i_1-i_2} \\
&& = B_2 \frac{i_1}{i_1+i_2} \frac{n_1!n_2!}{(n_1-i_1)!i_1!(n_2-i_2)!i_2!} (i_1+i_2)!\left(\frac{\mu_1}{\lambda}\right)^{i_1}\left(\frac{\mu_2}{\lambda}\right)^{i_2} \\
&& \sim B_3\left(\frac{i_1}{(i_1+i_2)(n_1-i_1)(n_2-i_2)i_2}\right)^{1/2} \frac{(i_1+i_2)^{i_1+i_2}\exp(-i_1-i_2)}
{(n_1-i_1)^{n_1-i_1}i_1^{i_1}(n_2-i_2)^{n_2-i_2}i_2^{i_2}}
\left(\frac{\mu_1}{\lambda}\right)^{i_1}\left(\frac{\mu_2}{\lambda}\right)^{i_2} \\
&& = B_3\left(\frac{i_1}{(i_1+i_2)(n_1-i_1)(n_2-i_2)i_2}\right)^{1/2} \exp\big((i_1+i_2)\log(i_1+i_2)-(n_1-i_1)\log(n_1-i_1)-i_1\log(i_1) \\
&&\qquad -(n_2-i_2)\log(n_2-i_2)-i_2\log(i_2)+i_1\log\left(\frac{\mu_1}{\lambda}\right)+i_2\log\left(\frac{\mu_2}{\lambda}\right)-i_1-i_2\big)
\end{eqnarray*}
where $B_3=B_2n_1!n_2!(2\pi)^{-\frac{3}{2}}e^{n}$.
Note that $\frac{(n_1-m_1)\mu_1}{\lambda}=\beta$ and $\frac{(n_2-m_2)\mu_2}{\lambda}=1-\beta$. Define $x_1:=\frac{i_1}{n}$ and $x_2:=\frac{i_2}{n}$, we have $x_1\in[f_1-\epsilon,f_1+\epsilon]$ and $x_2\in[f_2-\epsilon,f_2+\epsilon]$.
\begin{eqnarray*}
&&(i_1+i_2)\log(i_1+i_2)-(n_1-i_1)\log(n_1-i_1)-i_1\log(i_1)-(n_2-i_2)\log(n_2-i_2) -i_2\log(i_2)  \\
&&\qquad +i_1\log\left(\frac{\mu_1}{\lambda}\right)+i_2\log\left(\frac{\mu_2}{\lambda}\right)-i_1-i_2 \\
&&=(i_1+i_2)\log(i_1+i_2)-(n_1-i_1)\log(n_1-i_1)-(n_2-i_2)\log(n_2-i_2) \\
&&\qquad +i_1\log\left(\frac{\beta}{(n_1-m_1)i_1}\right)+i_2\log\left(\frac{1-\beta}{(n_2-m_2)i_2}\right)-i_1-i_2\\
&&=(i_1+i_2)\log(x_1+x_2)+(i_1+i_2)\log n-(n_1-i_1)\log(\theta-x_1)-(n_1-i_1)\log n-(n_2-i_2)\log(1-\theta-x_2)\\
&&\qquad -(n_2-i_2)\log n +i_1\log\left(\frac{\beta}{(\theta-f_1)x_1}\right)-2i_1\log n +i_2\log\left(\frac{1-\beta}{(1-\theta-f_2)x_2}\right)-2i_2\log n -i_1-i_2\\
&&= n((x_1+x_2)\log(x_1+x_2)-(\theta-x_1)\log(\theta-x_1)-(1-\theta-x_2)\log(1-\theta-x_2)\\
&&\qquad +x_1(\log\beta-\log(\theta-f_1)-\log(x_1))+ x_2(\log(1-\beta)-\log(1-\theta-f_2)-\log(x_2))-\log n-x_1-x_2)
\end{eqnarray*}
We define
\begin{eqnarray*}
&&F(x_1,x_2):=(x_1+x_2)\log(x_1+x_2)-(\theta-x_1)\log(\theta-x_1)-(1-\theta-x_2)\log(1-\theta-x_2) \\
&&\quad +x_1(\log\beta-\log(\theta-f_1)-\log(x_1))+x_2(\log(1-\beta)-\log(1-\theta-f_2)-\log(x_2))-x_1-x_2
\end{eqnarray*}
The first order derivatives on $x_1$ and $x_2$:
\begin{eqnarray*}
&&\frac{\partial F}{\partial x_1} = \log(x_1+x_2)+\log(\theta-x_1)-\log(x_1)-\log\frac{\beta}{\theta-f_1}=0\\
&&\frac{\partial F}{\partial x_2} = \log(x_1+x_2)+\log(1-\theta-x_2)-\log(x_2)+\log\frac{1-\beta}{1-\theta-f_2}=0
\end{eqnarray*}
Noting $\frac{f_1}{f_2}=\frac{\beta}{1-\beta}$, we can verify that
\[
x_1=f_1, \qquad x_2=f_2
\]
solve the first order conditions. Look at the second order derivatives:
\begin{eqnarray*}
&\frac{\partial^2 F}{\partial x_1^2}&= -\frac{1}{\theta-x_1} -\frac{1}{x_1}+\frac{1}{x_1+x_2}<0\\
&\frac{\partial^2 F}{\partial x_2^2}&= -\frac{1}{1-\theta-x_2} -\frac{1}{x_2}+\frac{1}{x_1+x_2}<0\\
&\frac{\partial^2 F}{\partial x_1\partial x_2} &= \frac{1}{x_1+x_2}
\end{eqnarray*}
\[
\frac{\partial^2 F}{\partial x_1^2}\frac{\partial^2 F}{\partial x_2^2}-\left(\frac{\partial^2 F}{\partial x_1 \partial x_2}\right)^2 =\frac{x_1^2(1-\theta)+x_2^2\theta}{x_1x_2(x_1+x_2)(\theta-x_1)(1-\theta-x_2)}> 0
\]
The Hessian matrix is negative definite. Therefore, $F(x_1,x_2)$ is strictly concave on $(0,\theta)\times(0,1-\theta)$ and reaches its unique global maximum at $(f_1,f_2)$.  Since $F(x_1,x_2)$ is strictly concave and reaches its unique global maximum at $(f_1,f_2)$. The maximum of $F(x_1,x_2)$ on $[\delta,\theta-\delta]\times [\delta, 1-\theta-\delta]\backslash (f_1-\epsilon,f_1+\epsilon)\times(f_2-\epsilon,f_2+\epsilon)$ is on the boundary $\{(x_1,x_2)\vert \lvert x_1-f_1\rvert=\epsilon,\lvert x_2-f_2\rvert=\epsilon\}$. Since the boundary is a compact set, the maximum is attainable, denoted by $F(f_1,f_2)-\eta$, where $\eta>0$.

Note that
\[
\left(\frac{i_1}{(i_1+i_2)(n_1-i_1)(n_2-i_2)i_2}\right)^{1/2} = \left(\frac{x_1}{(x_1+x_2)(\theta-x_1)(1-\theta-x_2)x_2}\right)^{1/2}n^{-1}
\]
changes slowly when $x_1$ and $x_2$ change, compared with $\exp\left(nF(x_1,x_2)\right)$. We have
\[
\frac{P(I_1(\frak{s})=i_1,I_2(\frak{s})=i_2|K(\frak{s})=0)}
{P(I_1(\frak{s})= \lceil m_1\rceil, I_2(\frak{s})= \lceil m_2\rceil|K(\frak{s})=0)}\sim \exp\left(n(F(x_1,x_2)-F(f_1,f_2))\right)<\exp(\eta n).
\]

Therefore,
\begin{eqnarray*}
\frac{\sum_{|i_1-m_1|>\epsilon n \mbox{ or } |i_2-m_2|>\epsilon n}P(I_1(\frak{s})=i_1,I_2(\frak{s})=i_2|K(\frak{s})=0)}
{P(I_1(\frak{s})= \lceil m_1\rceil, I_2(\frak{s})= \lceil m_2\rceil|K(\frak{s})=0)} <\left( n_1 n_2  + \frac{(i_1^*+1)(i_2^*+1)}{\beta^{i_1^*+1}(1-\beta)^{i_2^*+1}}\right)\exp(\eta n)
\end{eqnarray*}
It converges to 0 when $n\to\infty$.

When $K(\frak{s})=k>0$, and $n\to\infty$, similarly,
\[
P(I_1(\frak{s})=i_1, I_2(\frak{s})=i_2|K(\frak{s})=k)
 = B_1{n_1\choose i_1}{n_2\choose i_2}i_1 (i_1+i_2-k-1)!\frac{i_2!}{(i_2-k)!} \mu_1^{i_1}\mu_2^{i_2}\lambda^{-i_1-i_2-k}\lambda_1^k\big/ P(K(\frak{s})=k)
\]
\[
\frac{P(I_1(\frak{s})=i_1+1, I_2(\frak{s})=i_2|K(\frak{s})=k)}{P(I_1(\frak{s})=i_1, I_2(\frak{s})=i_2|K(\frak{s})=k)}
=\frac{(i_1+i_2-k)(n_1-i_1)\mu_1}{i_1\lambda}=\beta\frac{n_1-i_1}{n_1-m_1}\frac{i_1+i_2-k}{i_1}.
\]
\[
\frac{P(I_1(\frak{s})=i_1, I_2(\frak{s})=i_2+1|K(\frak{s})=k)}{P(I_1(\frak{s})=i_1, I_2(\frak{s})=i_2|K(\frak{s})=k)}
=\frac{(i_1+i_2-k)(n_2-i_2)\mu_2}{(i_2+1-k)\lambda}=(1-\beta)\frac{n_2-i_2}{n_2-m_2}\frac{i_1+i_2-k}{i_2+1-k}.
\]
We can use the similar 2-step argument to show that $(I_1(\frak{s})/n,\,I_2(\frak{s})/n)$ converges to $(f_1,\,f_2)$ in probability given $K(\frak{s})=k$.
\end{proof}


\begin{proof}[Proof of Theorem \ref{thm.IdleDistribution}]
To obtain the asymptotic distribution of $I_1,I_2$ as $n\to \infty$ we need to consider, by Theorem~\ref{thm.FluidLimit}, only values $i_1,i_2$ for which $(i_1 - m_1)/n \to 0$ and $(i_2-m_2)/n\to 0$.  We write $i_1=m_1+z_1\sqrt{n},\,i_2=m_2+z_2\sqrt{n}$, with $z_1/ \sqrt{n} \to 0$, $z_2/ \sqrt{n} \to 0$.  Note that
$m_1,\,m_2,\,n_1-m_1,\,n_2-m_2$ are of  the same order of magnitude as $n,n_1,n_2$, and we only consider $i_1,i_2$ of the same order of magnitude.
\begin{eqnarray*}
&&P(I_1(\frak{s})=i_1, I_2(\frak{s})=i_2|K(\frak{s})=0) \\
&&\qquad = B_2\frac{i_1}{i_1+i_2}  \frac{n_1!n_2!}{(n_1-i_1)!i_1!(n_2-i_2)!i_2!} (i_1+i_2)!\mu_1^{i_1}\mu_2^{i_2}\lambda^{-i_1-i_2}\\
&&\qquad \sim B_3i_1^{-i_1}(n_1-i_1)^{-(n_1-i_1)}i_2^{-i_2}(n_2-i_2)^{-(n_2-i_2)}\left(\frac{i_1+i_2}{e}\right)^{i_1+i_2}
          \mu_1^{i_1}\mu_2^{i_2}\lambda^{-i_1-i_2}\\
&&\qquad  \quad\times  \left(\frac{i_1}{(i_1+i_2) i_2 (n_1-i_1)(n_2-i_2)}\right)^{1/2},
\end{eqnarray*}
where the use of Stirling's approximation is justified for large $n$.  Here $B_2=B_1/P(K(\frak{s})=0)$ and $B_3=B_2n_1!n_2!(2\pi)^{-\frac{3}{2}}e^{n}$.

We clearly have:
\[
\left(\frac{i_1}{(i_1+i_2) i_2 (n_1-i_1)(n_2-i_2)}\right)^{1/2} \sim \left(\frac{m_1}{(m_1+m_2) m_2 (n_1-m_1)(n_2-m_2)}\right)^{1/2}.
\]
So we can treat that part as a constant.
Consider
\[
i_1^{-i_1}=(m_1+z_1\sqrt{n})^{-(m_1+z_1\sqrt{n})}=m_1^{-(m_1+z_1\sqrt{n})}\left(1+\frac{z_1\sqrt{n}}{m_1}\right)^{-(m_1+z_1\sqrt{n})}, \]
then from the Taylor expansion of the logarithm function, we have
\begin{eqnarray*}
&&\log\left(\left(1+\frac{z_1\sqrt{n}}{m_1}\right)^{-(m_1+z_1\sqrt{n})}\right) = -(m_1+z_1\sqrt{n})\log\left(1+\frac{z_1\sqrt{n}}{m_1}\right) \\
&&=-(m_1+z_1\sqrt{n})\left(\frac{z_1\sqrt{n}}{m_1}-\frac{z_1^2 n}{2m_1^2}+o\left(\frac{1}{n}\right)\right)
=-z_1\sqrt{n} - \frac{z_1^2n}{2m_1^2} + o(1)
\end{eqnarray*}
Therefore,
\[
\log(i_1^{-i_1})\approx -(m_1+z_1\sqrt{n})\log(m_1)-z_1\sqrt{n} - \frac{z_1^2n}{2m_1}
\]
Similar expansions are valid for $n_1-i_1,\,i_2,\,n_2-i_2$ and $i_1+i_2$.

Therefore, we have
\begin{eqnarray*}
&& \log P(I_1(\frak{s})=i_1, I_2(\frak{s})=i_2|K(\frak{s})=0)   \\
&&\sim \log(B_4)  +z_1\sqrt{n}\log\left(\frac{(n_1-m_1)\mu_1}{\lambda m_1}\right)-\frac{z_1^2nn_1}{2(n_1-m_1)m_1} \\
&&\qquad +z_2\sqrt{n}\log\left(\frac{(n_2-m_2)\mu_2}{\lambda m_2}\right)-\frac{z_2^2nn_2}{2(n_2-m_2)m_2} \\
&&\qquad +(z_1+z_2)\sqrt{n}\log(m_1+m_2) + \frac{(z_1+z_2)^2n}{2(m_1+m_2)},
\end{eqnarray*}
where $B_4= B_3 \left(\frac{m_1}{(m_1+m_2) m_2 (n_1-m_1)(n_2-m_2)}\right)^{1/2}$.

We now use the calculations in Section \ref{sec.fluid} to evaluate all the $\sqrt{n}$ coefficients. By (\ref{eqn.MeanIdle}) we have
\[
\frac{(n_1-m_1)\mu_1}{\lambda m_1} = \frac{1}{\lambda T}, \quad \frac{(n_2-m_2)\mu_2}{\lambda m_2} = \frac{1}{\lambda T},
\]
\[
m_1+m_2 = T \left(\frac{n_1}{T+1/\mu_1}+\frac{n_2}{T+1/\mu_2}\right) = \lambda T
\]
Therefore,
\[
z_1\sqrt{n}\log\left(\frac{(n_1-m_1)\mu_1}{\lambda m_1}\right) + z_2\sqrt{n}\log\left(\frac{(n_2-m_2)\mu_2}{\lambda m_2}\right) + (z_1+z_2)\sqrt{n}\log(m_1+m_2) = 0
\]
We are left with
\begin{equation*}
P(I_1(\frak{s})=i_1, I_2(\frak{s})=i_2|K(\frak{s})=0) \sim \exp(B_4)\exp\left( \frac{(z_1+z_2)^2n}{2(m_1+m_2)}-\frac{z_1^2nn_1}{2(n_1-m_1)m_1}-\frac{z_2^2nn_2}{2(n_2-m_2)m_2}\right).
\end{equation*}
Define
\[
\rho = \left(\frac{(n_1-m_1)(n_2-m_2)m_1m_2}{(n_1m_2+m_1^2)(n_2m_1+m_2^2)}\right)^{\frac{1}{2}},
\]
\[
\sigma_1 = \left(\frac{(n_1-m_1)m_1(n_2m_1+m_2^2)}{n_1m_2^2+n_2m_1^2}\right)^{\frac{1}{2}},
\]
\[
\sigma_2 = \left(\frac{(n_2-m_2)m_2(n_1m_2+m_1^2)}{n_1m_2^2+n_2m_1^2}\right)^{\frac{1}{2}}.
\]
We have
\begin{equation*}
P(I_1(\frak{s})=i_1, I_2(\frak{s})=i_2|K(\frak{s})=0) \sim \exp(B_4)
\exp\left(-\frac{1}{2(1-\rho^2)}\left(\frac{z_1^2n}{\sigma_1^2}+\frac{z_2^2n}{\sigma_2^2}-\frac{2\rho z_1z_2n}{\sigma_1\sigma_2}\right)\right),
\end{equation*}
Therefore,
$(\frac{I_1(\frak{s})-m_1}{\sqrt{n}},\frac{I_2(\frak{s})-m_2}{\sqrt{n}})$ given $K(\frak{s})=0$ converges in distribution as $n\to \infty$  to the bivariate Normal distribution as stated in (\ref{eqn.clt}).

When $K(\frak{s})=k>0$, and $n\to\infty$, similarly,
\begin{eqnarray*}
&&P(I_1(\frak{s})=i_1, I_2(\frak{s})=i_2|K(\frak{s})=k) \\
&& \qquad  = B_1{n_1\choose i_1}{n_2\choose i_2}i_1 (i_1+i_2-k-1)!\frac{i_2!}{(i_2-k)!} \mu_1^{i_1}\mu_2^{i_2}\lambda^{-i_1-i_2-k}\lambda_1^k \big/ P(K(\frak{s})=k)\\
&& \quad \sim B_1 \alpha^{k}\frac{i_1\,i_2^k}{(i_1+i_2)^{k+1}} \frac{n_1!n_2!}{(n_1-i_1)!i_1!(n_2-i_2)!i_2!} (i_1+i_2)!\mu_1^{i_1}\mu_2^{i_2}\lambda^{-i_1-i_2}\big/ P(K(\frak{s})=k).\\
\end{eqnarray*}
We again write $i_1=m_1+z_1\sqrt{n},\,i_2=m_2+z_2\sqrt{n}$, with $z_1/ \sqrt{n} \to 0$, $z_2/ \sqrt{n} \to 0$.  We then have
\[
\frac{i_1\,i_2^k}{(i_1+i_2)^{k+1}}\to \frac{m_1\,m_2^k}{(m_1+m_2)^{k+1}}=\beta(1-\beta)^k.
\]
We can now use the same approximation as for $k=0$ to show that $\left(\frac{I_1(\frak{s})-m_1}{\sqrt{n}},\,\frac{I_2(\frak{s})-m_2}{\sqrt{n}}\right)$ converge to the same bivariate Normal distribution.
\end{proof}

\subsection{Proof of Theorem \ref{thm.Kdist}}
\begin{proof}[Proof of Theorem \ref{thm.Kdist}]
Take a fixed arbitrary $\epsilon\in\left(0,\min\{f_1,f_2\}\right)$. Fix $k>0$, for any $i_1,i_2$ satisfying $|i_1/n-f_1|<\epsilon$, $|i_2/n-f_2|<\epsilon$ and $i_1\geq 1$, from (\ref{eqn.piKI}), noting $\frac{a+c}{b+c}\geq\frac{a}{b}$ for any $0<a\leq b$ and $c>0$, we have
\begin{equation} \label{eq:upperbound}
\begin{aligned}
\frac{\pi(k, i_1,i_2)}{\pi(k-1,i_1,i_2)} &= \frac{i_2-k+1}{i_1+i_2-k}\frac{1}{\alpha}\leq
\frac{i_2+1}{i_1+i_2}\frac{1}{\alpha}\leq\frac{(f_2+\epsilon)n+1}{i_1+(f_2+\epsilon)n}\frac{1}{\alpha}
<\frac{(f_2+\epsilon)n+1}{(f_1-\epsilon)n+(f_2+\epsilon)n}\frac{1}{\alpha} \\
&=\frac{(f_2+\epsilon)n+1}{(f_1+f_2)n}\frac{1}{\alpha}
=\frac{f_2}{(f_1+f_2)\alpha}\left(1+\frac{\epsilon}{f_2}+\frac{1}{f_2n}\right)
=\frac{1-\beta}{\alpha}\left(1+\frac{\epsilon}{f_2}+\frac{1}{f_2n}\right)
\end{aligned}
\end{equation}
Therefore,
\[
\pi(k,i_1,i_2) <\pi(0,i_1,i_2)\left(\frac{1-\beta}{\alpha}\right)^{k} \left(1+\frac{\epsilon}{f_2}+\frac{1}{f_2n}\right)^{k}.
\]
For fixed $k_0>0$,
\[
P\left(K(\frak{s})\geq k_0,I_1(\frak{s})=i_1,I_2(\frak{s})=i_2\right)
<\pi(0,i_1,i_2)\frac{\left(\frac{1-\beta}{\alpha}\right)^{k_0}
\left(1+\frac{\epsilon}{f_2}+\frac{1}{f_2n}\right)^{k_0}}{1-\left(\frac{1-\beta}{\alpha}\right)\left(1+\frac{\epsilon}{f_2}+\frac{1}{f_2n}\right)}.
\]
Note the above inequality is valid for any $i_1,i_2$ satisfying $|i_1/n-f_1|<\epsilon$, $|i_2/n-f_2|<\epsilon$, we have
\[
\frac{P\left(K(\frak{s})\geq k_0,|I_1(\frak{s})/n-f_1|<\epsilon,|I_2(\frak{s})-f_2|<\epsilon\right)}{P\left(K(\frak{s})=0,|I_1(\frak{s})/n-f_1|<\epsilon,|I_2(\frak{s})-f_2|<\epsilon\right)}
<\frac{\left(\frac{1-\beta}{\alpha}\right)^{k_0}
\left(1+\frac{\epsilon}{f_2}+\frac{1}{f_2n}\right)^{k_0}}{1-\left(\frac{1-\beta}{\alpha}\right)\left(1+\frac{\epsilon}{f_2}+\frac{1}{f_2n}\right)}.
\]
From Theorem \ref{thm.FluidLimit}, there exists an $N_1$ such that when $n>N_1$,
\[
P(|I_1(\frak{s})/n-f_1| < \epsilon, |I_2(\frak{s})/n-f_2| < \epsilon )>1-\epsilon.
\]
Then we have,
\begin{align*}
P(K(\frak{s}) \geq k_0)& <P\left(\left.K(\frak{s})\geq k_0\right||I_1(\frak{s})/n-f_1|\geq \epsilon,|I_2(\frak{s})/n-f_2|\geq \epsilon \right) P\left(|I_1(\frak{s})/n-f_1|< \epsilon,|I_2(\frak{s})/n-f_2|< \epsilon \right)\\
&\qquad+ P\left(|I_1(\frak{s})/n-f_1|\geq \epsilon,|I_2(\frak{s})/n-f_2|\geq \epsilon\right)\\
&< P\left(\left.K(\frak{s})=0\right||I_1(\frak{s})/n-f_1|\geq \epsilon,|I_2(\frak{s})/n-f_2|\geq \epsilon \right)
\frac{\left(\frac{1-\beta}{\alpha}\right)^{k_0}
\left(1+\frac{\epsilon}{f_2}+\frac{1}{f_2n}\right)^{k_0}}{1-\left(\frac{1-\beta}{\alpha}\right)\left(1+\frac{\epsilon}{f_2}+\frac{1}{f_2n}\right)} (1-\epsilon) + \epsilon \\
&< \frac{\left(\frac{1-\beta}{\alpha}\right)^{k_0}
\left(1+\frac{\epsilon}{f_2}+\frac{1}{f_2n}\right)^{k_0}}{1-\left(\frac{1-\beta}{\alpha}\right)\left(1+\frac{\epsilon}{f_2}+\frac{1}{f_2n}\right)} (1-\epsilon) + \epsilon.
\end{align*}
This upper bound can be arbitrarily close to 0 when choosing $\epsilon$, $n>N_1$ and $k_0$. Therefore, we have shown the tightness of $K(\frak{s})$, that is
\begin{equation}\label{eq:tightness}
\sum_{k=0}^\infty \lim_{n\to\infty}P\left(K(\frak{s})=k\right) = 1.
\end{equation}
Using
\[
P(K(\frak{s})=k) = P\left(K(\frak{s})=k,|I_1(\frak{s})/n-f_1|<\epsilon,|I_2(\frak{s})/n-f_2|<\epsilon\right)
+P\left(K(\frak{s})=k,|I_1(\frak{s})/n-f_1|\geq\epsilon,|I_2(\frak{s})/n-f_2|\geq\epsilon\right),
\]
for fixed $k>0$, when $n>N_1$, the ratio $\frac{P(K(\frak{s})=k)}{P(K(\frak{s})=k-1)}$ is lower bounded by
\[\frac{P\left(K(\frak{s})=k,|I_1(\frak{s})/n-f_1|< \epsilon,|I_2(\frak{s})/n-f_2|<\epsilon\right)}
{P\left(K(\frak{s})=k-1,|I_1(\frak{s})/n-f_1|< \epsilon,|I_2(\frak{s})/n-f_2|<\epsilon\right)+\epsilon}
\]
and upper bounded by
\[
\frac{P\left(K(\frak{s})=k,|I_1(\frak{s})/n-f_1|<\epsilon,|I_2(\frak{s})/n-f_2|<\epsilon\right)+\epsilon}
{P\left(K(\frak{s})=k-1,|I_1(\frak{s})/n-f_1|< \epsilon,|I_2(\frak{s})/n-f_2|<\epsilon\right)}.
\]
For any $i_1,i_2$ satisfying $|i_1/n-f_1|<\epsilon$, $|i_2/n-f_2|<\epsilon$ and $i_1\geq 1$, in addition to (\ref{eq:upperbound}), we have the lower bound
\begin{align*}
\frac{\pi(k, i_1,i_2)}{\pi(k-1,i_1,i_2)}= \frac{i_2-k+1}{i_1+i_2-k}\frac{1}{\alpha}\geq\frac{(f_2-\epsilon)n-k+1}{i_1+(f_2-\epsilon)n-k}\frac{1}{\alpha}
>\frac{(f_2-\epsilon)n-k+1}{(f_1+\epsilon)n+(f_2-\epsilon)n-k}\frac{1}{\alpha} =\frac{(f_2-\epsilon)n-k+1}{(f_1+f_2)n-k}\frac{1}{\alpha}.
\end{align*}
Now we have
\[
\frac{\pi(k, i_1,i_2)}{\pi(k-1,i_1,i_2)} \in \left[\frac{(f_2-\epsilon)n-k+1}{(f_1+f_2)n-k}\frac{1}{\alpha}, \frac{(f_2+\epsilon)n+1}{(f_1+f_2)n}\frac{1}{\alpha} \right].
\]
Therefore,
\[
\frac{\sum_{|i_1/n-f_1|< \epsilon,|i_2/n-f_2|<\epsilon}\pi(k,i_1,i_2)}{\sum_{|i_1/n-f_1|< \epsilon,|i_2/n-f_2|<\epsilon}\pi(k-1,i_1,i_2)}\in \left[\frac{(f_2-\epsilon)n-k+1}{(f_1+f_2)n-k}\frac{1}{\alpha}, \frac{(f_2+\epsilon)n+1}{(f_1+f_2)n}\frac{1}{\alpha} \right],
\]
that is,
\begin{equation}\label{eq:bounds}
\frac{P\left(K(\frak{s})=k,|I_1(\frak{s})/n-f_1|< \epsilon,|I_2(\frak{s})/n-f_2|<\epsilon\right)}
{P\left(K(\frak{s})=k-1,|I_1(\frak{s})/n-f_1|< \epsilon,|I_2(\frak{s})/n-f_2|<\epsilon\right)}\in \left[\frac{(f_2-\epsilon)n-k+1}{(f_1+f_2)n-k}\frac{1}{\alpha}, \frac{(f_2+\epsilon)n+1}{(f_1+f_2)n}\frac{1}{\alpha} \right].
\end{equation}
For fixed $k$, as $n\to\infty$, the lower bound and the upper bound in (\ref{eq:bounds}) both converge to $\frac{1-\beta}{\alpha}$. Noting $\epsilon$ can be arbitrarily close to 0, we have
\[
\lim_{n\to\infty} \frac{P(K(\frak{s})=k)}{P(K(\frak{s})=k-1)} = \frac{1-\alpha}{\beta}
\]
This together with the tightness (\ref{eq:tightness}) proves (\ref{eqn.limK}).
\end{proof}

\end{document}